\documentclass[final,onefignum,onetabnum]{siamart220329}
\usepackage{braket,amsfonts}

\usepackage{array}

\usepackage[caption=false]{subfig}

\newsiamthm{claim}{Claim}
\newsiamremark{rmk}{Remark}
\newsiamthm{thm}{Theorem}
\newsiamthm{df}{Definition}
\newsiamthm{prf}{Proof} 
\newsiamthm{ex}{Example}
\newsiamthm{prop}{Proposition}
\newsiamthm{lm}{Lemma}
\newsiamthm{cor}{Corollary}
\usepackage{algorithmic}
\usepackage{graphicx,epstopdf}

\Crefname{ALC@unique}{Line}{Lines}

\usepackage{amsopn}

\usepackage{xspace}
\usepackage{bold-extra}
\usepackage[most]{tcolorbox}

\colorlet{texcscolor}{blue!50!black}
\colorlet{texemcolor}{red!70!black}
\colorlet{texpreamble}{red!70!black}
\colorlet{codebackground}{black!25!white!25}

\lstdefinestyle{siamlatex}{%
  style=tcblatex,
  texcsstyle=*\color{texcscolor},
  texcsstyle=[2]\color{texemcolor},
  keywordstyle=[2]\color{texemcolor},
  moretexcs={cref,Cref,maketitle,mathcal,text,headers,email,url},
}

\tcbset{%
  colframe=black!75!white!75,
  coltitle=white,
  colback=codebackground, 
  colbacklower=white, 
  fonttitle=\bfseries,
  arc=0pt,outer arc=0pt,
  top=1pt,bottom=1pt,left=1mm,right=1mm,middle=1mm,boxsep=1mm,
  leftrule=0.3mm,rightrule=0.3mm,toprule=0.3mm,bottomrule=0.3mm,
  listing options={style=siamlatex}
}

\newtcblisting[use counter=example]{example}[2][]{%
  title={Example~\thetcbcounter: #2},#1}

\newtcbinputlisting[use counter=example]{\examplefile}[3][]{%
  title={Example~\thetcbcounter: #2},listing file={#3},#1}

\DeclareTotalTCBox{\code}{ v O{} }
{ 
  fontupper=\ttfamily\color{black},
  nobeforeafter,
  tcbox raise base,
  colback=codebackground,colframe=white,
  top=0pt,bottom=0pt,left=0mm,right=0mm,
  leftrule=0pt,rightrule=0pt,toprule=0mm,bottomrule=0mm,
  boxsep=0.5mm,
  #2}{#1}

\patchcmd\newpage{\vfil}{}{}{}
\begin{tcbverbatimwrite}{tmp_\jobname_header.tex}
\title{Null Controllability for Stochastic Parabolic Equations with convection terms}
\author{Mahmoud Baroun\thanks{Cadi Ayyad University, Faculty of Sciences Semlalia, LMDP, UMMISCO (IRD-UPMC), P.B. 2390, Marrakesh, Morocco (\email{m.baroun@uca.ac.ma}, \email{abdellatif.elgrou@ced.uca.ma}, \email{maniar@uca.ma}).}
\and Said Boulite\thanks{Cadi Ayyad University, National School of Applied Sciences, LMDP, UMMISCO (IRD-UPMC), P.B. 575, Marrakesh, Morocco (\email{s.boulite@uca.ma}).}
\and Abdellatif Elgrou\footnotemark[1]
\and Lahcen Maniar\footnotemark[1]}
\headers{Carleman, Stochastic parabolic equations, And Controllability}{M. Baroun, S. Boulite, A. Elgrou, and L. Maniar}

\end{tcbverbatimwrite}
\input{tmp_\jobname_header.tex}

\ifpdf
\hypersetup{ pdftitle={Carleman and null controllability for stoch parab eq \LaTeX\ Style} }
\fi
\usepackage{amssymb}
\usepackage{bbm}
\usepackage{hyperref}
\usepackage{amsmath}

\begin{document}
\maketitle

\begin{tcbverbatimwrite}{tmp_\jobname_abstract.tex}
\begin{abstract}
This paper addresses null
controllability for both forward and backward linear stochastic parabolic equations by introducing convection terms on the drift parts with bounded coefficients. Moreover, the forward stochastic parabolic equation includes a convection term on the diffusion part. The null controllability results rely on novel Carleman estimates for both backward and forward stochastic parabolic equations, encompassing a divergence source term interpreted in the weak sense. These Carleman estimates are established through the application of the duality technique. In doing so, we resolve some previously unanswered questions (see Remarks 2.1-2.2 in [S. Tang, and X. Zhang, SIAM J. Control Optim., 48 (2009), p.p 2191-2216]). Additionally, we present a more accurate estimation of the null-control costs.
\end{abstract}



\end{tcbverbatimwrite}
\input{tmp_\jobname_abstract.tex}

\section{Introduction and main results}
Let $T>0$, $G\subset\mathbb{R}^N$ (with $N\geq1$) be a given nonempty bounded domain with smooth boundary $\Gamma=\partial G$, and let $G_0\Subset G$ be a given non-empty open subset which is strictly contained in $G$ (i.e., $\overline{G_0}\subset G$) and  $\overline{G}$ denotes the closure of $G$.  Put $Q=(0,T)\times G, \,\,\, \Sigma=(0,T)\times\Gamma,\,\, \textnormal{and}\,\,\, Q_0=(0,T)\times G_0$. Also, we indicate by $\mathbbm{1}_{G_0}$ the characteristic function of $G_0$. 

Let $(\Omega,\mathcal{F},\{\mathcal{F}_t\}_{t\in[0,T]},\mathbb{P})$ be a fixed complete filtered probability space on which a one-dimensional standard Brownian motion $W(\cdot)$ is defined such that $\{\mathcal{F}_t\}_{t\in[0,T]}$ is the natural filtration generated by $W(\cdot)$ and augmented by all the $\mathbb{P}$-null sets in $\mathcal{F}$. Let $\mathcal{H}$ be a Banach space, let $C([0,T];\mathcal{H})$ be the Banach space of all $\mathcal{H}$-valued continuous functions defined on $[0,T]$; and for some sub-sigma algebra $\mathcal{G}\subset\mathcal{F}$, we denote by $L^2_{\mathcal{G}}(\Omega;\mathcal{H})$ the Banach space of all $\mathcal{H}$-valued $\mathcal{G}$-measurable random variables $X$ such that $\mathbb{E}\big(\vert X\vert_\mathcal{H}^2\big)<\infty$, with the canonical norm; by $L^2_\mathcal{F}(0,T;\mathcal{H})$ the Banach space consisting of all $\mathcal{H}$-valued $\{\mathcal{F}_t\}_{t\in[0,T]}$-adapted processes $X(\cdot)$ such that $\mathbb{E}\big(\vert X(\cdot)\vert^2_{L^2(0,T;\mathcal{H})}\big)<\infty$, with the canonical norm; by $L^\infty_\mathcal{F}(0,T;\mathcal{H})$ the Banach space consisting of all $\mathcal{H}$-valued $\{\mathcal{F}_t\}_{t\in[0,T]}$-adapted essentially bounded processes and its norm is denoted simply by $|\cdot|_\infty$; and by $L^2_\mathcal{F}(\Omega;C([0,T];\mathcal{H}))$ the Banach space consisting of all $\mathcal{H}$-valued $\{\mathcal{F}_t\}_{t\in[0,T]}$-adapted continuous processes $X(\cdot)$ such that $\mathbb{E}\big(\vert X(\cdot)\vert^2_{C([0,T];\mathcal{H})}\big)<\infty$, with the canonical norm. Similarly, one can define $L^\infty_\mathcal{F}(\Omega;C^m([0,T];\mathcal{H}))$ for any positive integer $m$.

The first purpose of this paper is to study the null controllability of the following forward linear stochastic parabolic equation
\begin{equation}\label{1.1}
\begin{cases}
\begin{array}{ll}
dy - \nabla\cdot(A\nabla y) \,dt = (a_1 y+B_1\cdot\nabla y +\mathbbm{1}_{G_0}u) \,dt \\
\hspace{3.1cm}\;+ (a_2y+B_2\cdot\nabla y+v) \,dW(t) &\textnormal{in}\,\,Q,\\
				y=0 &\textnormal{on}\,\,\Sigma,\\
				y(0)=y_0 &\textnormal{in}\,\,G,
			\end{array}
		\end{cases}
\end{equation}
where $a_1, a_2\in L_\mathcal{F}^\infty(0, T; L^\infty(G))$, $B_1, B_2\in L_\mathcal{F}^\infty(0, T;L^\infty(G;\mathbb{R}^N))$, $y_0\in L^2_{\mathcal{F}_0}(\Omega;L^2(G))$ is the initial state and the control function consists of the pair 
$$(u,v)\in L^2_\mathcal{F}(0,T;L^2(G_0))\times L^2_\mathcal{F}(0,T;L^2(G)).$$ 
For $\nabla$ (resp., $\nabla\cdot$) represents gradient (resp., divergence). The term $\nabla\cdot(A\nabla y)$ (resp., $B_i\cdot\nabla y$, $i=1,2$) describes diffusion (resp., convection terms). We assume that the diffusion matrix $A=(a^{jk})_{j,k=1,2,...,N}$ satisfies
 \begin{enumerate}
\item $A\in L^\infty_\mathcal{F}(\Omega;C^1([0,T];W^{2,\infty}(G;\mathbb{R}^{N\times N})))$, and $a^{jk}=a^{kj}$\,\, for all \,\,$j,k=1,2,...,N$.
\item There exists a constant $\beta>0$ such that
    $$( A(\omega,t,x)\xi,\xi)_{\mathbb{R}^N} \geq \beta\vert\xi\vert^2,\qquad (\omega,t,x,\xi)\in\Omega\times Q\times\mathbb{R}^N.$$
\end{enumerate}
Here and in the sequel of this paper, we denote by ``\,$\cdot$\,'' or $(\cdot,\cdot)_{\mathbb{R}^N}$ the canonical inner product in $\mathbb{R}^N$, and $C$ stands for a positive constant depending only on $G$, $G_0$, and $A$, which may change from one place to another.

From \cite{krylov1994}, it is well-known that \eqref{1.1} is well-posed i.e., for any 
$y_0\in L^2_{\mathcal{F}_0}(\Omega;L^2(G))$ and $(u,v)\in L^2_\mathcal{F}(0,T;L^2(G_0))\times L^2_\mathcal{F}(0,T;L^2(G))$, the equation \eqref{1.1} has a unique weak solution 
$$y\in L^2_\mathcal{F}(\Omega;C([0,T];L^2(G)))\bigcap L^2_\mathcal{F}(0,T;H^1_0(G)),$$
that depends continuously on $y_0$, $u$ and $v$.

The first main result of this paper is the following null controllability of \eqref{1.1}.
\begin{thm}\label{thm01.1}
For any given $T>0$, $G_0\Subset G$ a nonempty open subset of $G$, and for all $y_0\in L^2_{\mathcal{F}_0}(\Omega;L^2(G))$, there exist controls $(\hat{u},\hat{v})\in L^2_\mathcal{F}(0,T;L^2(G_0))\times L^2_\mathcal{F}(0,T;L^2(G))$ such that the corresponding solution $\hat{y}$ to \eqref{1.1} satisfies  $$\hat{y}(T,\cdot) = 0\,\,\, \textnormal{in}\,\,G,\,\,\mathbb{P}\textnormal{-a.s.}$$
Moreover, controls $\hat{u}$ and $\hat{v}$ can be chosen in such a way that
\begin{align}\label{1.2201}
\begin{aligned}
\vert \hat{u}\vert^2_{L^2_\mathcal{F}(0,T;L^2(G))}+\vert \hat{v}\vert^2_{ L^2_\mathcal{F}(0,T;L^2(G))}\leq e^{CK}\,\vert y_0\vert^2_{L^2_{\mathcal{F}_0}(\Omega;L^2(G))},
\end{aligned}
\end{align}
where $K=K(T,\vert a_1\vert_\infty,\vert a_2\vert_\infty,\vert B_1\vert_\infty,\vert B_2\vert_\infty)$ has the following form
\begin{align*}
K\equiv1+\frac{1}{T}+\vert a_1\vert_\infty^{2/3}+T\vert a_1\vert_\infty+\vert a_2\vert_\infty^{2/3}+T\vert a_2\vert_\infty^{2}+(1+T)\vert B_1\vert_\infty^2+\vert B_2\vert_\infty^2.
\end{align*}
\end{thm} 

The second purpose of this paper is to establish the null controllability of the following backward linear stochastic parabolic equation
	\begin{equation}\label{1.2}
		\begin{cases}
			\begin{array}{ll}
			dy+\nabla\cdot(A\nabla y) dt = (a_1y+B\cdot\nabla y+a_2Y+\mathbbm{1}_{G_0}u)dt + Y dW(t)	 &\textnormal{in}\,\,Q,\\
			y=0 &\textnormal{on}\,\,\Sigma,\\
			y(T)=y_T &\textnormal{in}\,\,G,
			\end{array}
		\end{cases}
	\end{equation}
where $A$ is as previously, $a_1, a_2\in L^\infty_\mathcal{F}(0,T;L^\infty(G))$, $B\in L^\infty_\mathcal{F}(0,T;L^\infty(G;\mathbb{R}^N))$, $y_T\in L^2_{\mathcal{F}_T}(\Omega;L^2(G))$ is the terminal state and $u\in L^2_\mathcal{F}(0,T;L^2(G_0))$ is the control process.

According to \cite{zhou1992}, equation \eqref{1.2} is well-posed i.e., for any $y_T\in L^2_{\mathcal{F}_T}(\Omega;L^2(G))$ and $u\in L^2_\mathcal{F}(0,T;L^2(G_0))$, there exists a unique weak solution $(y,Y)$ of \eqref{1.2} such that
$$(y,Y)\in \Big(L^2_\mathcal{F}(\Omega;C([0,T];L^2(G)))\bigcap L^2_\mathcal{F}(0,T;H^1_0(G))\Big)\times L^2_\mathcal{F}(0,T;L^2(G)),$$
and that depends continuously on $y_T$ and $u$.

The second main result is the following null controllability of \eqref{1.2}.
\begin{thm}\label{thm1.2}
For any given $T>0$, $G_0\Subset G$ a nonempty open subset of $G$, and for all $y_T\in L^2_{\mathcal{F}_T}(\Omega;L^2(G))$, there exists a control $\hat{u}\in L^2_\mathcal{F}(0,T;L^2(G_0))$ such that the corresponding solution $\hat{y}$ to \eqref{1.2} satisfies  $$\hat{y}(0,\cdot) = 0\,\,\, \textnormal{in}\,\,G,\,\,\mathbb{P}\textnormal{-a.s.}$$ Furthermore, the control $\hat{u}$ can be chosen such that
\begin{equation}\label{1.44012}
\vert \hat{u}\vert^2_{L^2_\mathcal{F}(0,T;L^2(G))}\leq e^{CM}\,\vert y_T\vert^2_{L^2_{\mathcal{F}_T}(\Omega;L^2(G))},
\end{equation}
where $M=M(T,\vert a_1\vert_\infty,\vert a_2\vert_\infty,\vert B\vert_\infty)$ is given by
$$M\equiv1+\frac{1}{T}+\vert a_1\vert_\infty^{2/3}+T\vert a_1\vert_\infty+(1+T)\vert a_2\vert_\infty^2+(1+T)\vert B\vert_\infty^2.$$
\end{thm} 
\begin{rmk}
Equations \eqref{1.1} (with $u\equiv v\equiv0$) and \eqref{1.2} (with $u\equiv 0$) combine the diffusion and convection (or advection) equations and are widely used to describe many physical phenomena such as fluid evolution with a particular speed. They also represent the change in a moving medium of some quantities such as heat, bacteria population, chemical concentration, etc. Under the stochastic perturbation, they can consider all small independent random disturbances that can happen along the evolution of such quantities.
\end{rmk}
\begin{rmk}
The null controllability results of equations \eqref{1.1} and \eqref{1.2} with convection terms appearing only on the drift parts under the regularity $W^{1,\infty}$ in space are already established in \cite{tang2009null}. In this paper, we study the null controllability of \eqref{1.1} and \eqref{1.2} by assuming that the coefficients of convection terms are only $L^\infty$ in space. In such a study, to prove our Carleman estimates, we employ the duality technique as described in \cite{imanuvilov2003carleman}, used for the first time in the stochastic setting in \cite{liu2014global} to eliminate an extra gradient term in the right-hand side of the previous Carleman estimate for forward stochastic parabolic equations proved in \cite{tang2009null}. In our case, we utilize the duality technique to deal with the problem of the weak divergence terms that arise in the adjoint equations \eqref{1.3} and \eqref{1.5}. Moreover, the study of the null controllability result of \eqref{1.1} with the general diffusion term ``$[a_2y+B_2\cdot\nabla y]dW(t)$'' was suggested in \cite[Remark 2.1]{tang2009null}.
\end{rmk}

The key point to establish the null controllability of \eqref{1.1} is to obtain an appropriate observability inequality for its adjoint backward stochastic parabolic equation 
\begin{equation}\label{1.3}
\begin{cases}
\begin{array}{ll}
dz + \nabla\cdot(A\nabla z) \,dt = (-a_1 z-a_2Z+\nabla\cdot(zB_1+ZB_2)) \,dt + Z \,dW(t) &\textnormal{in}\,\,Q,\\
z=0 &\textnormal{on}\,\,\Sigma,\\
z(T)=z_T &\textnormal{in}\,\,G,
			\end{array}
		\end{cases}
\end{equation}
where $z_T\in L^2_{\mathcal{F}_T}(\Omega;L^2(G))$. From \cite{zhou1992}, one can deduce the well-posedness result of equation \eqref{1.3} i.e., for all $z_T\in L^2_{\mathcal{F}_T}(\Omega;L^2(G))$, there exists a unique weak solution $$(z,Z)\in \Big(L^2_\mathcal{F}(\Omega;C([0,T];L^2(G)))\bigcap L^2_\mathcal{F}(0,T;H^1_0(G))\Big)\times L^2_\mathcal{F}(0,T;L^2(G)),$$
which depends continuously on $z_T$.  In section \ref{sec4}, we will prove the following observability inequality: For all $z_T\in L^2_{\mathcal{F}_T}(\Omega;L^2(G))$, the corresponding solution $(z,Z)$ to \eqref{1.3} satisfies
\begin{equation}\label{1.04}
\vert z(0,\cdot)\vert^2_{L^2_{\mathcal{F}_0}(\Omega;L^2(G))} \leq e^{CK}\,\Bigg[\mathbb{E}\int_{Q_0} z^2 dxdt+\mathbb{E}\int_{Q} Z^2 dxdt\Bigg],
\end{equation}    
with $K$ is as in Theorem \ref{thm01.1} (in fact, the positive constant $C$ appearing in \eqref{1.2201} is the same constant that appears in \eqref{1.04}). 

To establish the null controllability result of \eqref{1.2}, we introduce the following adjoint forward stochastic parabolic equation
	\begin{equation}\label{1.5}
		\begin{cases}
			\begin{array}{ll}
			dz-\nabla\cdot(A\nabla z) dt = (-a_1z+\nabla\cdot(zB))dt - a_2z dW(t)	 &\textnormal{in}\,\,Q,\\
			z=0 &\textnormal{on}\,\,\Sigma,\\
			z(0)=z_0 &\textnormal{in}\,\,G,
			\end{array}
		\end{cases}
	\end{equation}
 where $z_0\in L^2_{\mathcal{F}_0}(\Omega;L^2(G))$. Similarly to \eqref{1.1}, we have the following well-posedness of \eqref{1.5}: For any $z_0\in L^2_{\mathcal{F}_0}(\Omega;L^2(G))$, there exists a unique weak solution $$z\in L^2_\mathcal{F}(\Omega;C([0,T];L^2(G)))\bigcap L^2_\mathcal{F}(0,T;H^1_0(G)),$$
that depends continuously on $z_0$. 
In this case, we are interested in the following observability problem of \eqref{1.5} i.e., for all $z_0\in L^2_{\mathcal{F}_0}(\Omega;L^2(G))$, the associated solution $z$ of \eqref{1.5} fulfills that 
\begin{equation}\label{1.8}
    \vert z(T,\cdot)\vert^2_{L^2_{\mathcal{F}_T}(\Omega;L^2(G))} \leq e^{CM}\,\mathbb{E}\int_{Q_0} z^2 dxdt,
\end{equation}
with $M$ is similar to the one defined in Theorem \ref{thm1.2} (indeed, the constant $C$ appearing in \eqref{1.44012} is also similar to the constant that appears in \eqref{1.8}).

Now, several remarks are in order.
\begin{rmk}
Inequality \eqref{1.2201} (resp., \eqref{1.44012}) gives an explicit estimate of the cost of null controllability of equation \eqref{1.1} (resp., \eqref{1.2}) which blows up when $T\rightarrow0$. We refer to \cite{fernandezuazua,seidman} for some related works on the control cost in the deterministic setting.
\end{rmk}
\begin{rmk}
The estimate \eqref{1.04} (resp., \eqref{1.8}) provides the observability inequality for \eqref{1.3} (resp., \eqref{1.5}) with an explicit observability constant (see section \ref{sec4}), depending on the control time $T$, and potentials $a_1, a_2, B_1$ and $B_2$ (resp., $a_1, a_2$ and $B$). In \cite{tang2009null}, when $T\in(0,1)$, the observability constants have the form $Ce^{CT^{-4}}$, however in our situation, they will have the form $Ce^{CT^{-1}}$, which gives sharp observability estimates. At this level, it would be interesting to study the optimality of such observability constants (see, e.g., \cite{optilobsinq} for some related works in the deterministic case)
\end{rmk}
\begin{rmk}
From Carleman estimate \eqref{1.014} (resp., \eqref{3.202002}), it is straightforward to deduce the unique continuation property for solutions of equation \eqref{1.3} (resp., \eqref{1.5}). Therefore, utilizing the duality argument, one can also establish the approximate controllability results for equations \eqref{1.1} and \eqref{1.2}.
\end{rmk}
\begin{rmk}
The question of whether \eqref{1.1} is null controllable by using only the control $u$ in the drift term, or at least acting as well by only a localized (on a subdomain of $G$) control $v$ in the diffusion part, is still an open and interesting problem to solve. Currently, we employ Carleman estimates with (deterministic) weight functions that only depend on time and space variables. To handle this problem, we may need (random) weight functions. In the case when the coefficients are space-independent, using the spectral method, the author in \cite{lu2011some} gives some interesting results on the null and approximate controllability of forward stochastic heat equations with only one control force on the drift term. 
\end{rmk}

In the existing literature, one can find many controllability results for certain types of linear stochastic parabolic equations (see, e.g., \cite{barbu2003carleman,SPEwithDBC,SPEwithDBC1D,liu2014global,liu2019carleman,lu2021mathematical,fourthorder,lu2011some,tang2009null,yan2018carleman,yansing} and the references cited therein). Moreover, for some results concerning the null controllability of semilinear stochastic parabolic equations, we refer to \cite{san23}. In our study, we prove new Carleman estimates for stochastic parabolic equations containing a weak divergence source term. These estimates will allow us to establish the null controllability results for equations \eqref{1.1} and \eqref{1.2}. Note that the authors in
\cite{tang2009null} established Carleman estimates by using an exponentially weighted energy identity for a stochastic parabolic-like operator. In our situation, due to the presence of divergence terms that appear in the adjoint equations and that such terms can only be understood in a weak sense, we are not able to use the same weighted identity as in \cite{tang2009null}. Therefore, we adopt the duality technique to find the appropriate Carleman estimates for stochastic parabolic equations with weak divergence source terms. Dealing with this, when we want to derive an energy estimate or the duality relation between the controlled equations \eqref{1.1} and \eqref{1.2} and their adjoint equations, we will need the following Itô's formula for a weak form of processes, see, e.g., \cite[Chapter 2]{lu2021mathematical} for more details and proofs.
\begin{lm}\label{lm1.1}
Let $z,y\in L^2_\mathcal{F}(0,T;H^1_0(G))$, $z_T\in L^2_{\mathcal{F}_T}(\Omega;L^2(G))$, $y_0\in L^2_{\mathcal{F}_0}(\Omega;L^2(G))$, $\phi,\Tilde{\phi}\in L^2_\mathcal{F}(0,T;H^{-1}(G))$ and $\Psi,Z\in L^2_\mathcal{F}(0,T;L^2(G))$, such that for all $t\in [0,T]$, the processes $(z,Z)$ and $y$ satisfying the following equations
{\small\begin{align*}
z(t)=z_T-\int_t^T \phi(s)ds-\int_t^T Z(s)dW(s),\,\,\mathbb{P}\textnormal{-a.s.},\\
y(t)=y_0+\int_0^t \Tilde{\phi}(s)ds+\int_0^t \Psi(s)dW(s),\,\,\mathbb{P}\textnormal{-a.s.}
\end{align*}}
Then
\begin{enumerate}
\item For any $t\in[0,T]$, it holds that
\begin{align*}
\vert z(t)\vert^2_{L^2(G)}=&\,\vert z(0)\vert^2_{L^2(G)}+2\int_0^t \langle\phi,z\rangle_{H^{-1}(G),H^1_0(G)}ds\\
&+2\int_0^t (z,Z)_{L^2(G)}dW(s)+\int_0^t \vert Z\vert_{L^2(G)}^2ds,\;\mathbb{P}\textnormal{-a.s.}
\end{align*}
\item For all $t\in[0,T]$, we have
\begin{align*}
( z(t),y(t))_{L^2(G)}=&\,( z(0),y_0)_{L^2(G)}+\int_0^t \langle\phi(s),y(s)\rangle_{H^{-1}(G),H^1_0(G)}ds\\
&+\int_0^t \langle\Tilde{\phi}(s),z(s)\rangle_{H^{-1}(G),H^1_0(G)}ds+\int_0^t (y(s),Z(s))_{L^2(G)}dW(s)\\
&+\int_0^t (z(s),\Psi(s))_{L^2(G)}dW(s)+\int_0^t (Z(s),\Psi(s))_{L^2(G)}ds,\;\mathbb{P}\textnormal{-a.s.}
\end{align*}
\end{enumerate}
\end{lm}
In the above lemma, $(\cdot,\cdot)_{L^2(G)}$ and $\langle\cdot,\cdot\rangle_{H^{-1}(G),H^1_0(G)}$ denote respectively, the inner product in $L^2(G)$, and the duality product between $H^1_0(G)$ and its dual space $H^{-1}(G)$ w.r.t the pivot space $L^2(G)$.

Note also that the operator $\nabla\cdot$ appearing in \eqref{1.3} and \eqref{1.5} is to be understood in the weak sense, given by the following lemma.
\begin{lm}\label{lm1.2}
For any $\eta\in L^2(G;\mathbb{R}^N)$, we have the following extension of the operator divergence of $\eta$ as the linear continuous operator on $H^1_0(G)$ given by
\begin{equation*}
\nabla\cdot\eta:H^1_0(G)\longrightarrow\mathbb{R},\;\;\; v \longmapsto -\int_G \eta\cdot\nabla v\, dx.
\end{equation*}
\end{lm}

The rest of this paper is organized as follows. Section \ref{sec2} (resp., \ref{sec3}) is devoted to establishing global Carleman estimates for general backward (resp., forward) stochastic parabolic equations with a weak divergence source term. In section \ref{sec4}, we apply such Carleman estimates to show the observability inequalities \eqref{1.04} and \eqref{1.8}. The proof of Theorems \ref{thm01.1} and \ref{thm1.2} is given in section \ref{sec5}. Finally, we have included an appendix where we complete the proof of a known Carleman estimate for backward stochastic parabolic equations by carefully examining the dependence of Carleman's parameters on time $T$. Such an estimate will be essential in our analysis in section \ref{sec2}.
\section{Global Carleman estimate for backward stochastic parabolic equations}\label{sec2}
Due to Imanuvilov, we have the following technical lemma.
\begin{lm}\label{lm2.1}
For any nonempty open subset $G_1\Subset G$, there exists a function $\psi\in C^4(\overline{G})$ such that
\begin{equation}
 \psi(x)>0 \,\,\textnormal{in}\,\,G,\quad\quad\psi(x)=0\,\,\textnormal{on}\,\,\Gamma,\quad\quad \vert\nabla\psi(x)\vert>0\,\,\textnormal{in}\,\,\overline{G\setminus G_1}.
\end{equation}
\end{lm}
The existence of such function $\psi$ in Lemma \ref{lm2.1} is proved in \cite{fursikov1996controllability}. For any (large) parameters $\lambda, \mu\geq1$, we define the following weight functions
\begin{equation}\label{2.2012}
\begin{array}{l}
\varphi(t,x) = (t(T-t))^{-1}\,e^{\mu\psi(x)},\quad \theta=e^l,\quad l=\lambda\alpha,\\\\
\alpha(t,x) = (t(T-t))^{-1}\,(e^{\mu\psi(x)}-e^{2\mu\vert\psi\vert_\infty}).
\end{array}
\end{equation}
It is easy to check that for some suitable constant $C>0$, we have for all $(t,x)\in Q$ 
	\begin{equation}\label{2.301}
		\begin{array}{l}
	\varphi(t,x)\geq CT^{-2},\qquad\vert\varphi_t(t,x)\vert\leq CT\varphi^2(t,x),\qquad\vert\varphi_{tt}(t,x)\vert\leq CT^2\varphi^3(t,x),\\\\
			\vert\alpha_t(t,x)\vert\leq CTe^{2\mu\vert\psi\vert_\infty}\varphi^2(t,x),\qquad\vert\alpha_{tt}(t,x)\vert\leq CT^2e^{2\mu\vert\psi\vert_\infty}\varphi^3(t,x).
		\end{array}
	\end{equation}
 
We now consider the following backward linear stochastic parabolic equation 
\begin{equation}\label{202.3}
\begin{cases}
\begin{array}{ll}
dz + \nabla\cdot(A\nabla z) \,dt = (F_0+\nabla\cdot F) \,dt + Z \,dW(t) &\textnormal{in}\,\,Q,\\
z=0 &\textnormal{on}\,\,\Sigma,\\
z(T)=z_T &\textnormal{in}\,\,G,
			\end{array}
		\end{cases}
\end{equation}
where $z_T\in L^2_{\mathcal{F}_T}(\Omega;L^2(G))$, $F_0\in L^2_\mathcal{F}(0,T;L^2(G))$ and $F\in L^2_\mathcal{F}(0,T;L^2(G;\mathbb{R}^N))$. 

The main result of this section is the following global Carleman estimate for \eqref{202.3}.
\begin{thm}\label{thm02.1}
There exists a constant $C=C(G,G_0,A)>0$ so that for all $z_T\in L^2_{\mathcal{F}_T}(\Omega;L^2(G))$, $F_0\in L^2_\mathcal{F}(0,T;L^2(G))$ and $F\in L^2_\mathcal{F}(0,T;L^2(G;\mathbb{R}^N))$, the weak solution $(z,Z)$ of \eqref{202.3} satisfies
\begin{align}\label{1.014}
\begin{aligned}
&\,\lambda^3\mu^4\mathbb{E}\int_Q \theta^2\varphi^3z^2dxdt+\lambda\mu^2\mathbb{E}\int_Q \theta^2\varphi\vert\nabla z\vert^2dxdt\\
&\leq  C\Bigg[\lambda^3\mu^4\mathbb{E}\int_{Q_0} \theta^2\varphi^3z^2dxdt+\mathbb{E}\int_Q \theta^2 F_0^2dxdt\\
&\quad\quad\;\,+\lambda^2\mu^2\mathbb{E}\int_Q \theta^2\varphi^2\vert F\vert^2dxdt+\lambda^2\mu^2\mathbb{E}\int_Q \theta^2\varphi^2Z^2dxdt\Bigg],
\end{aligned}
\end{align}
for all $\mu\geq C$ and $\lambda\geq C(e^{2\mu\vert\psi\vert_\infty}T+T^2)$.
\end{thm}

To prove Theorem \ref{thm02.1}, we will need the following auxiliary result which is a Carleman estimate for equation \eqref{202.3} (with $F\equiv0$). For the reader’s convenience, we have sketched the proof of this result in the appendix, at the end of this paper.
\begin{lm}\label{lm2.2}
There exists $C=C(G,G_0,A)>0$ so that for all $z_T\in L^2_{\mathcal{F}_T}(\Omega;L^2(G))$, $F_0\in L^2_\mathcal{F}(0,T;L^2(G))$, the weak solution $(z,Z)$ of \eqref{202.3} $($with $F\equiv0$$)$ satisfies
\begin{align}
\begin{aligned}
&\,\lambda^3\mu^4\mathbb{E}\int_Q \theta^2\varphi^3z^2dxdt+\lambda\mu^2\mathbb{E}\int_Q \theta^2\varphi\vert\nabla z\vert^2dxdt\\
\label{2.061}&\leq  C\Bigg[\lambda^3\mu^4\mathbb{E}\int_{Q_0} \theta^2\varphi^3z^2dxdt+\mathbb{E}\int_Q \theta^2 F_0^2dxdt+\lambda^2\mu^2\mathbb{E}\int_Q \theta^2\varphi^2Z^2dxdt\Bigg],
\end{aligned}
\end{align}
for $\mu\geq C$ and $\lambda\geq C(e^{2\mu\vert\psi\vert_\infty}T+T^2)$.
\end{lm}

Now, let us consider the following controlled forward stochastic parabolic equation
\begin{equation}\label{2.1}
\begin{cases}
\begin{array}{ll}
dy - \nabla\cdot(A\nabla y) \,dt = (\lambda^3\mu^4\theta
^2\varphi^3z+\mathbbm{1}_{G_0}u) \,dt + (\lambda^2\mu^2\theta
^2\varphi^2Z+v) \,dW(t) &\textnormal{in}\,\,Q,\\
y=0 &\textnormal{on}\,\,\Sigma,\\
y(0)=0 &\textnormal{in}\,\,G,
			\end{array}
		\end{cases}
\end{equation}
where the pair $(u,v)\in L^2_\mathcal{F}(0,T;L^2(G_0))\times L^2_\mathcal{F}(0,T;L^2(G))$ is the control of the system, $(z,Z)$ is the solution of equation \eqref{202.3} and $\theta$ and $\varphi$ are the weight functions defined in \eqref{2.2012}. Based on the Carleman estimate in the above Lemma, we first establish the following result on the null controllability and a Carleman estimate of system \eqref{2.1} by using the duality technique.

\begin{prop}\label{prop2.4}
There exists $(\hat{u},\hat{v})\in L^2_\mathcal{F}(0,T;L^2(G_0))\times L^2_\mathcal{F}(0,T;L^2(G))$ a pair of controls such that the associated solution $\hat{y}\in L^2_\mathcal{F}(0,T;H^1_0(G))$ to \eqref{2.1} satisfies 
$$\hat{y}(T,\cdot)=0 \;\;\textnormal{in}\;\, G, \;\,\mathbb{P}\textnormal{-a.s.}$$ 
Moreover, there exists $C=C(G,G_0,A)>0$ such that 
\begin{align}
\begin{aligned}
&\,\mathbb{E}\int_Q \theta^{-2}\hat{y}^2dxdt+\lambda^{-2}\mu^{-2}\mathbb{E}\int_Q \theta^{-2}\varphi^{-2}\vert\nabla\hat{y}\vert^2dxdt\\
&+\lambda^{-3}\mu^{-4}\mathbb{E}\int_{Q} \theta^{-2}\varphi^{-3}\hat{u}^2dxdt+\lambda^{-2}\mu^{-2}\mathbb{E}\int_Q \theta^{-2}\varphi^{-2}\hat{v}^2dxdt\\
\label{2.2}&
\leq C\Bigg[\lambda^3\mu^4\mathbb{E}\int_Q \theta^2\varphi^3z^2dxdt + \lambda^2\mu^2\mathbb{E}\int_Q \theta^2\varphi^2 Z^2dxdt\Bigg],
\end{aligned}
\end{align}
for all $\mu\geq C$ and $\lambda\geq C(e^{2\mu\vert\psi\vert_\infty}T+T^2)$.
\end{prop}
\begin{proof}
Let $\varepsilon>0$, define $\theta_\varepsilon=e^{\lambda\alpha_\varepsilon}$ with $$\alpha_\varepsilon\equiv\alpha_\varepsilon(t,x)=((t+\varepsilon)(T-t+\varepsilon))^{-1}(e^{\mu\psi(x)}-e^{2\mu\vert\psi\vert_\infty}),$$ and note that $\theta_\varepsilon\geq\theta$. Define the following minimization problem
\begin{equation}\label{2.03}
    \inf\big\{ J_\varepsilon(u,v)\,\vert\, u,v\in\mathcal{U}\big\},
\end{equation}
where 
\begin{align*}
 J_\varepsilon(u,v)=&\,\frac{1}{2}\mathbb{E}\int_Q \theta^{-2}_\varepsilon y^2dxdt+\frac{1}{2}\mathbb{E}\int_{Q_0} \lambda^{-3}\mu^{-4}\theta^{-2}\varphi^{-3}u^2dxdt\\
 &+\frac{1}{2}\mathbb{E}\int_Q \lambda^{-2}\mu^{-2}\theta^{-2}\varphi^{-2}v^2dxdt+\frac{1}{2\varepsilon}\mathbb{E}\int_G \vert y(T)\vert^2dx,
 \end{align*}
and
\begin{align*}
\mathcal{U}=\Bigg\{(&u,v)\in L^2_\mathcal{F}(0,T;L^2(G_0))\times L^2_\mathcal{F}(0,T;L^2(G)):\\
&\mathbb{E}\int_{Q_0} \theta^{-2}\varphi^{-3}u^2dxdt<\infty\,,\quad\mathbb{E}\int_Q \theta^{-2}\varphi^{-2}v^2dxdt<\infty\Bigg\}.
\end{align*}
First, it is easy to see that $J_\varepsilon$ is continuous, strictly convex, and coercive, then the problem \eqref{2.03} admits a unique optimal solution $(u_\varepsilon,v_\varepsilon)\in\mathcal{U}$. Moreover, by the standard duality argument (see, for instance, \cite{imanuvilov2003carleman,lions1972some}), such optimal solution can be characterized as
\begin{equation}\label{2.003}
 u_\varepsilon=-\mathbbm{1}_{G_0}\lambda^3\mu^4\theta^2\varphi^3r_\varepsilon\,\,,\quad\quad \quad v_\varepsilon=-\lambda^2\mu^2\theta^2\varphi^2R_\varepsilon,
 \end{equation}
with $(r_\varepsilon,R_\varepsilon)$ is the solution of the following backward stochastic parabolic equation
\begin{equation}\label{2.04}
\begin{cases}
\begin{array}{ll}
dr_\varepsilon + \nabla\cdot(A\nabla r_\varepsilon) \,dt = -\theta^{-2}_\varepsilon y_\varepsilon \,dt + R_\varepsilon \,dW(t) &\textnormal{in}\,\,Q,\\
r_\varepsilon=0 &\textnormal{on}\,\,\Sigma,\\
r_\varepsilon(T)=\frac{1}{\varepsilon}y_\varepsilon(T) &\textnormal{in}\,\,G,
			\end{array}
		\end{cases}
\end{equation}
where $y_\varepsilon$ is the solution of \eqref{2.1} with controls $u_\varepsilon$ and $v_\varepsilon$. Computing $d(y_\varepsilon r_\varepsilon)$ by Itô's formula and using \eqref{2.003}, we get
\begin{align}
\begin{aligned}
&\,\frac{1}{\varepsilon}\mathbb{E}\int_G \vert y_\varepsilon(T)\vert^2dx+\lambda^3\mu^4\mathbb{E}\int_{Q_0}\theta^2\varphi^3r_\varepsilon^2dxdt\\
&+\lambda^2\mu^2\mathbb{E}\int_{Q}\theta^2\varphi^2R_\varepsilon^2dxdt+\mathbb{E}\int_{Q}\theta_\varepsilon^{-2}y_\varepsilon^2dxdt\\
\label{2.005}&
= \lambda^3\mu^4\mathbb{E}\int_{Q}\theta^2\varphi^3 zr_\varepsilon dxdt+\lambda^2\mu^2\mathbb{E}\int_{Q}\theta^2\varphi^2Z R_\varepsilon dxdt.
\end{aligned}
\end{align}
Applying Young's inequality in the right-hand side of \eqref{2.005}, we have for all $\rho>0$
\begin{align}
\begin{aligned}
&\,\frac{1}{\varepsilon}\mathbb{E}\int_G \vert y_\varepsilon(T)\vert^2dx+\lambda^3\mu^4\mathbb{E}\int_{Q_0}\theta^2\varphi^3r_\varepsilon^2dxdt\\
&+\lambda^2\mu^2\mathbb{E}\int_{Q}\theta^2\varphi^2R_\varepsilon^2dxdt+\mathbb{E}\int_{Q}\theta_\varepsilon^{-2}y_\varepsilon^2dxdt\\&
\leq \rho\Bigg[\lambda^3\mu^4\mathbb{E}\int_{Q}\theta^2\varphi^3r_\varepsilon^2dxdt+\lambda^2\mu^2\mathbb{E}\int_{Q}\theta^2\varphi^2R_\varepsilon^2dxdt\Bigg]\\
\label{2.00151402}&\quad+C(\rho)\Bigg[\lambda^3\mu^4\mathbb{E}\int_{Q}\theta^2\varphi^3z^2dxdt+\lambda^2\mu^2\mathbb{E}\int_{Q}\theta^2\varphi^2Z^2dxdt\Bigg].
\end{aligned}
\end{align}
By using \eqref{2.061} for equation \eqref{2.04} and the fact that $\theta^2\theta_\varepsilon^{-2}\leq1$, the estimate \eqref{2.00151402} implies that for $\mu\geq C$ and $\lambda\geq C(e^{2\mu\vert\psi\vert_\infty}T+T^2)$, one has
\begin{align}
\begin{aligned}
&\,\frac{1}{\varepsilon}\mathbb{E}\int_G \vert y_\varepsilon(T)\vert^2+\lambda^3\mu^4\mathbb{E}\int_{Q_0}\theta^2\varphi^3r_\varepsilon^2dxdt\\
&+\lambda^2\mu^2\mathbb{E}\int_{Q}\theta^2\varphi^2R_\varepsilon^2dxdt+\mathbb{E}\int_{Q}\theta_\varepsilon^{-2}y_\varepsilon^2dxdt\\
&\leq C\rho\Bigg[\lambda^3\mu^4\mathbb{E}\int_{Q_0}\theta^2\varphi^3r_\varepsilon^2dxdt+\mathbb{E}\int_Q \theta_\varepsilon^{-2}y_\varepsilon^2 dxdt+\lambda^2\mu^2\mathbb{E}\int_{Q}\theta^2\varphi^2R_\varepsilon^2dxdt\Bigg]\\
\label{2.0015}&
\quad+C(\rho)\Bigg[\lambda^3\mu^4\mathbb{E}\int_{Q}\theta^2\varphi^3z^2dxdt+\lambda^2\mu^2\mathbb{E}\int_{Q}\theta^2\varphi^2Z^2dxdt\Bigg].
\end{aligned}
\end{align}
Taking a small enough $\rho>0$ in \eqref{2.0015} and recalling \eqref{2.003}, we obtain
\begin{align}
\begin{aligned}
&\,\lambda^{-3}\mu^{-4}\mathbb{E}\int_{Q}\theta^{-2}\varphi^{-3}u_\varepsilon^2dxdt+\lambda^{-2}\mu^{-2}\mathbb{E}\int_{Q}\theta^{-2}\varphi^{-2}v_\varepsilon^2dxdt\\
&+\mathbb{E}\int_{Q} \theta^{-2}_\varepsilon y_\varepsilon^2dxdt+\frac{1}{\varepsilon}\mathbb{E}\int_{G}\vert y_\varepsilon(T)\vert^2dx\\
\label{2.07}&
\leq C\,\Bigg[\lambda^3\mu^4\mathbb{E}\int_{Q}\theta^2\varphi^3z^2dxdt+\lambda^2\mu^2\mathbb{E}\int_{Q}\theta^2\varphi^2Z^2dxdt\Bigg].
\end{aligned}
\end{align}

On the other hand, by using Itô's formula, we compute $d(\theta^{-2}_\varepsilon\varphi^{-2}y_\varepsilon^2)$, integrating the equality on $Q$ and taking the expectation on both sides, we get
\begin{align}
\begin{aligned}
&\,2 \mathbb{E}\int_Q \theta^{-2}_\varepsilon\varphi^{-2}A\nabla y_\varepsilon\cdot\nabla y_\varepsilon dxdt\\
&=-2 \mathbb{E}\int_Q y_\varepsilon A\nabla y_\varepsilon\cdot\nabla(\theta_\varepsilon^{-2}\varphi^{-2}) dxdt+\mathbb{E}\int_Q (\theta^{-2}_\varepsilon\varphi^{-2})_ty_\varepsilon^2dxdt\\
&\quad\,+2\lambda^3\mu^4\mathbb{E}\int_Q \theta^{-2}_\varepsilon\theta^2\varphi y_\varepsilon z dxdt+2\mathbb{E}\int_{Q_0}\theta^{-2}_\varepsilon\varphi^{-2}y_\varepsilon u_\varepsilon dxdt\\
&\quad\,+\lambda^4\mu^4\mathbb{E}\int_Q \theta^{-2}_\varepsilon\theta^4\varphi^2Z^2dxdt+\mathbb{E}\int_Q \theta^{-2}_\varepsilon\varphi^{-2}v_\varepsilon^2dxdt\\
\label{2.08}&\quad\,+2\lambda^2\mu^2\mathbb{E}\int_Q \theta_\varepsilon^{-2}\theta^2 Zv_\varepsilon dxdt.
\end{aligned}
\end{align}
It is easy to check that for a large enough $\lambda\geq CT^2$, we have
\begin{equation}\label{2.0140}
\vert(\theta^{-2}_\varepsilon\varphi^{-2})_t\vert\leq CT\lambda e^{2\mu\vert\psi\vert_\infty}\theta_\varepsilon^{-2},\qquad\quad\vert\nabla(\theta^{-2}_\varepsilon\varphi^{-2})\vert\leq C\lambda\mu\theta^{-2}_\varepsilon\varphi^{-1}.
\end{equation}
Using inequalities \eqref{2.0140}, assumptions on $A$ and that $\theta_\varepsilon^{-1}\leq\theta^{-1}$, we conclude that for a large $\mu\geq C$ and $\lambda\geq C(T+T^2)$ 
\begin{align}
\begin{aligned}
&\,\mathbb{E}\int_Q \theta_\varepsilon^{-2}\varphi^{-2}\vert\nabla y_\varepsilon\vert^2dxdt\\
&\leq CT\lambda e^{2\mu\vert\psi\vert_\infty}\mathbb{E}\int_Q \theta_\varepsilon^{-2}y_\varepsilon^2dxdt+C\mathbb{E}\int_Q \theta^{-2}\varphi^{-2}v_\varepsilon^2dxdt\\
&\quad+C\lambda^2\mu^2\mathbb{E}\int_Q  \vert  Z\vert\vert v_\varepsilon\vert dxdt+C\lambda\mu\mathbb{E}\int_Q \theta_\varepsilon^{-2}\varphi^{-1}\vert y_\varepsilon\vert\vert\nabla y_\varepsilon\vert dxdt\\
\label{2.09}&\quad+C\lambda^3\mu^4\mathbb{E}\int_Q \theta_\varepsilon^{-1}\theta\varphi\vert z\vert\vert y_\varepsilon\vert dxdt+C\lambda^4\mu^4\mathbb{E}\int_Q \theta^2\varphi^2Z^2dxdt\\
&\quad+C\mathbb{E}\int_{Q}\theta_\varepsilon^{-1}\theta^{-1}\varphi^{-2}\vert y_\varepsilon\vert\vert u_\varepsilon\vert dxdt.
\end{aligned}
\end{align}
By applying Young's inequality on the right-hand side of \eqref{2.09}, one absorbs the gradient term in the fourth integral by using the left-hand side, and hence multiplying the inequality by $\lambda^{-2}\mu^{-2}$, we get for a large $\mu\geq C$ and $\lambda\geq C(T+T^2)$ that
\begin{align}\label{2.010}
\begin{aligned}
&\,\lambda^{-2}\mu^{-2}\mathbb{E}\int_Q \theta_\varepsilon^{-2}\varphi^{-2}\vert\nabla y_\varepsilon\vert^2dxdt\\
&\leq C\mathbb{E}\int_Q \theta_\varepsilon^{-2} y_\varepsilon^2 dxdt+C\lambda^2\mu^4\mathbb{E}\int_Q \theta^2\varphi^2z^2dxdt\\
&\quad
+C\lambda^{-3}\mu^{-4}\mathbb{E}\int_{Q} \theta^{-2}\varphi^{-3}u_\varepsilon^2dxdt+C\lambda^{-2}\mu^{-2}\mathbb{E}\int_Q \theta^{-2}\varphi^{-2}v_\varepsilon^2dxdt\\
&\quad+C\lambda^2\mu^2\mathbb{E}\int_Q \theta^{2}\varphi^2Z^2dxdt.
\end{aligned}
\end{align}
Combining \eqref{2.07} and \eqref{2.010}, we obtain for $\lambda\geq C(e^{2\mu\vert\psi\vert_\infty}T+T^2)$ and $\mu\geq C$
\begin{align}
\begin{aligned}
&\,\lambda^{-3}\mu^{-4}\mathbb{E}\int_{Q}\theta^{-2}\varphi^{-3}u_\varepsilon^2dxdt+\lambda^{-2}\mu^{-2}\mathbb{E}\int_{Q}\theta^{-2}\varphi^{-2}v_\varepsilon^2dxdt+\mathbb{E}\int_{Q} \theta^{-2}_\varepsilon y_\varepsilon^2dxdt\\
&+\lambda^{-2}\mu^{-2}\mathbb{E}\int_Q \theta_\varepsilon^{-2}\varphi^{-2}\vert\nabla y_\varepsilon\vert^2dxdt+\frac{1}{\varepsilon}\mathbb{E}\int_{G}\vert y_\varepsilon(T)\vert^2dx\\
\label{2.00119}&\leq C\,\Bigg[\lambda^3\mu^4\mathbb{E}\int_{Q}\theta^2\varphi^3z^2dxdt+\lambda^2\mu^2\mathbb{E}\int_{Q}\theta^2\varphi^2Z^2dxdt\Bigg].
\end{aligned}
\end{align}
It follows that there exists  
$$(\hat{u},\hat{v},\hat{y})\in L^2_\mathcal{F}(0,T;L^2(G_0))\times L^2_\mathcal{F}(0,T;L^2(G))\times L^2_\mathcal{F}(0,T;H^1_0(G)),$$
such that as $\varepsilon\rightarrow0$,
\begin{align}\label{2.020}
\begin{aligned}
&u_\varepsilon \longrightarrow \hat{u},\,\,\,\textnormal{weakly in}\,\,\,L^2((0,T)\times\Omega;L^2(G_0));\\
&v_\varepsilon \longrightarrow \hat{v},\,\,\,\textnormal{weakly in}\,\,\,L^2((0,T)\times\Omega;L^2(G));\\ 
&y_{\varepsilon} \longrightarrow \hat{y},\,\,\,\textnormal{weakly in}\,\,\,L^2((0,T)\times\Omega;H^1_0(G)).
\end{aligned}
\end{align}
		
Now, let us show that $\hat{y}$ is the solution of equation \eqref{2.1} associated with controls $\hat{u}$ and $\hat{v}$. To prove this fact, let us denote by $\widetilde{y}$ the unique solution to \eqref{2.1} with controls $\hat{u}$ and $\hat{v}$. For any process $f\in L^2_\mathcal{F}(0,T;L^2(G))$, let $(\phi,\Phi)$ be the solution of the following backward stochastic parabolic equation
\begin{equation*}
\begin{cases}
\begin{array}{lcl}
d\phi+\nabla\cdot(A\nabla\phi) \,dt = f \,dt+\Phi \,dW(t) &\textnormal{in}& Q,\\
     \phi=0&\textnormal{on}&\Sigma,\\
					\phi(T)=0&\textnormal{in}&G.
				\end{array}
			\end{cases}
		\end{equation*}
By Itô's formula, computing ``$d(\phi y_\varepsilon)-d(\phi\widetilde{y})$", integrating the equality on $Q$, taking the expectation on both sides and letting $\varepsilon\rightarrow0$, we get
\begin{equation*}
\mathbb{E}\int_Q (\hat{y}-\widetilde{y})f dxdt= 0.
\end{equation*}
It follows that $\hat{y}=\widetilde{y}$ in $Q$, $\mathbb{P}\textnormal{-a.s.}$ Then, we deduce that $\hat{y}$ is the solution of \eqref{2.1} with controls $\hat{u}$ and $\hat{v}$. Finally, combining the uniform estimate \eqref{2.00119} and the weak convergence result \eqref{2.020}, we conclude the null controllability result of \eqref{2.1} and the desired global Carleman estimate \eqref{2.2}.
\end{proof}
\begin{proof}[Proof of Theorem \ref{thm02.1}]
Let $(z,Z)$ be the solution of \eqref{202.3} and $\hat{y}$ be the solution of \eqref{2.1} associated to controls $(\hat{u},\hat{v})$ obtained in Proposition \ref{prop2.4}. By applying Itô's formula in Lemma \ref{lm1.1}, computing $d(\hat{y},z)_{L^2(G)}$, integrating the result on $(0,T)$ and taking the expectation on both sides, we get
\begin{align}
\begin{aligned}
0=&\,\lambda^3\mu^4\mathbb{E}\int_Q \theta^2\varphi^3z^2dxdt+\mathbb{E}\int_Q \mathbbm{1}_{G_0}\hat{u}zdxdt+\mathbb{E}\int_Q 
 F_0\hat{y}dxdt\\
\label{2.2601}&
-\mathbb{E}\int_Q F\cdot\nabla\hat{y}dxdt+\lambda^2\mu^2\mathbb{E}\int_Q \theta^2\varphi^2Z^2dxdt+\mathbb{E}\int_Q \hat{v}Z dxdt.
\end{aligned}
\end{align}
Then, it follows from \eqref{2.2601} that
\begin{align}
\begin{aligned}
&\,\lambda^3\mu^4\mathbb{E}\int_Q \theta^2\varphi^3z^2dxdt
+\lambda^2\mu^2\mathbb{E}\int_Q \theta^2\varphi^2Z^2dxdt\\
\label{2.27012}&=-\mathbb{E}\int_Q \big[\mathbbm{1}_{G_0}\hat{u}z+F_0\hat{y}-F\cdot\nabla\hat{y}+\hat{v}Z\big] dxdt.
\end{aligned}
\end{align}
Using Young's inequality in the right-hand side of \eqref{2.27012}, we have for all $\rho>0$
\begin{align}
\begin{aligned}
&\,\lambda^3\mu^4\mathbb{E}\int_Q \theta^2\varphi^3z^2dxdt
+\lambda^2\mu^2\mathbb{E}\int_Q \theta^2\varphi^2Z^2dxdt\\
&\leq \rho\Bigg[\lambda^{-3}\mu^{-4}\mathbb{E}\int_{Q} \theta^{-2}\varphi^{-3}\hat{u}^2dxdt+\mathbb{E}\int_Q\theta^{-2}\hat{y}^2dxdt\\
&\quad\quad\;+\lambda^{-2}\mu^{-2}\mathbb{E}\int_Q \theta^{-2}\varphi^{-2}\vert\nabla\hat{y}\vert^2dxdt+\lambda^{-2}\mu^{-2}\mathbb{E}\int_Q \theta^{-2}\varphi^{-2}\hat{v}^2dxdt\Bigg]\\
&\quad +C(\rho)\Bigg[\lambda^3\mu^4\mathbb{E}\int_{Q_0} \theta^2\varphi^3z^2dxdt
+\mathbb{E}\int_Q \theta^2 F_0^2dxdt\\
\label{2.3017}&\quad\quad\quad\quad\;\;\,+\lambda^2\mu^2\mathbb{E}\int_Q \theta^2\varphi^2\vert F\vert^2dxdt+\lambda^2\mu^2\mathbb{E}\int_Q \theta^2\varphi^2 Z^2dxdt\Bigg].
\end{aligned}
\end{align}
Now, by using the Carleman inequality \eqref{2.2}, the estimate \eqref{2.3017} provides that for a small enough $\rho>0$, we have
\begin{align}
\begin{aligned}
\lambda^3\mu^4\mathbb{E}\int_Q \theta^2\varphi^3z^2dxdt
\leq&\, C\,\Bigg[\lambda^3\mu^4\mathbb{E}\int_{Q_0} \theta^2\varphi^3z^2dxdt
+\mathbb{E}\int_Q \theta^2 F_0^2dxdt\\
\label{2.014}&\quad\;\;+\lambda^2\mu^2\mathbb{E}\int_Q \theta^2\varphi^2\vert F\vert^2dxdt+\lambda^2\mu^2\mathbb{E}\int_Q \theta^2\varphi^2 Z^2dxdt\Bigg],
\end{aligned}
\end{align}
for $\mu\geq C$ and $\lambda\geq C(e^{2\mu\vert\psi\vert_\infty}T+T^2)$.\\

On the other hand, computing $d(\theta^2\varphi z,z)_{L^2(G)}$ by using Itô's formula in Lemma \ref{lm1.1}, integrating the equality on $(0,T)$ and taking the expectation on both sides, we obtain
\begin{align}
\begin{aligned}
2\mathbb{E}\int_Q \theta^2\varphi A\nabla z\cdot\nabla z dxdt=&\,-\mathbb{E}\int_Q (\theta^2\varphi)_tz^2dxdt-2\mathbb{E}\int_Q zA\nabla z\cdot\nabla(\theta^2\varphi)dxdt\\
\label{2.31012}&\,
-2\mathbb{E}\int_Q \theta^2\varphi F_0z dxdt+2\mathbb{E}\int_Q zF\cdot\nabla(\theta^2\varphi)dxdt\\
&\,+2\mathbb{E}\int_Q \theta^2\varphi F\cdot\nabla z dxdt-\mathbb{E}\int_Q \theta^2\varphi Z^2dxdt.
\end{aligned}
\end{align}
It is easy to see that for $\lambda\geq CT^2$
\begin{equation}\label{2.32012}
    \vert (\theta^2\varphi)_t\vert\leq CT\lambda e^{2\mu\vert\psi\vert_\infty}\theta^2\varphi^3,\,\,\;\qquad\;\; \vert\nabla(\theta^2\varphi)\vert\leq C\lambda\mu\theta^2\varphi^2.
\end{equation}
By assumptions on $A$, inequalities \eqref{2.31012} and \eqref{2.32012} imply that for $\lambda\geq CT^2$, we have
\begin{align}
\begin{aligned}
\mathbb{E}\int_Q \theta^2\varphi\vert\nabla z\vert^2dxdt\leq&\, CT\lambda e^{2\mu\vert\psi\vert_\infty}\mathbb{E}\int_Q \theta^2\varphi^3z^2dxdt\\
&+C\lambda\mu \mathbb{E}\int_Q \theta^2\varphi^2\vert z\vert\vert\nabla z\vert dxdt\\
\label{2.017}&
    +C\mathbb{E}\int_Q \theta^2\varphi\vert F_0\vert\vert z\vert dxdt+C\lambda\mu\mathbb{E}\int_Q \theta^2\varphi^2\vert z\vert\vert F\vert dxdt\\
&+C\mathbb{E}\int_Q \theta^2\varphi\vert F\vert\vert\nabla z\vert dxdt+C\mathbb{E}\int_Q \theta^2\varphi Z^2dxdt.
\end{aligned}
\end{align}
By Young's inequality, \eqref{2.017} implies that for all $\rho>0$
\begin{align}
\begin{aligned}
&\,\mathbb{E}\int_Q \theta^2\varphi\vert\nabla z\vert^2dxdt\\
&\leq \rho\,\mathbb{E}\int_Q \theta^2\varphi\vert\nabla z\vert^2dxdt+C(\rho)\lambda^2\mu^2\mathbb{E}\int_Q \theta^2\varphi^3z^2dxdt\\
&\quad+C(\rho)\mathbb{E}\int_Q \theta^2\varphi\vert F\vert^2dxdt+CT\lambda e^{2\mu\vert\psi\vert_\infty}\mathbb{E}\int_Q \theta^2\varphi^3z^2dxdt\\
\label{2.018}&\quad+
C\lambda\mu^2\mathbb{E}\int_Q \theta^2\varphi^2 z^2dxdt+C\lambda^{-1}\mu^{-2}\mathbb{E}\int_Q \theta^2 F_0^2 dxdt
\\
&\quad+C\lambda^2\mu^2\mathbb{E}\int_Q \theta^2\varphi^3 z^2 dxdt+C\mathbb{E}\int_Q \theta^2\varphi\vert F\vert^2 dxdt+C\mathbb{E}\int_Q \theta^2\varphi Z^2dxdt.
\end{aligned}
\end{align}
By taking a small $\rho>0$, one absorbs the first term on the right-hand side of \eqref{2.018} by using the left-hand side, hence multiplying the obtained inequality by $\lambda\mu^2$, we get
\begin{align}
\begin{aligned}
\lambda\mu^2\mathbb{E}\int_Q \theta^2\varphi\vert\nabla z\vert^2dxdt\leq&\, C\lambda^3\mu^4\mathbb{E}\int_Q \theta^2\varphi^3z^2dxdt+C\lambda\mu^2\mathbb{E}\int_Q \theta^2\varphi\vert F\vert^2dxdt\\
&+CT\lambda^2\mu^2 e^{2\mu\vert\psi\vert_\infty}\mathbb{E}\int_Q \theta^2\varphi^3z^2dxdt\\
&+C\lambda^2\mu^4\mathbb{E}\int_Q \theta^2\varphi^2 z^2dxdt+C\mathbb{E}\int_Q \theta^2 F_0^2 dxdt\\
\label{2.01836}&+C\lambda\mu^2\mathbb{E}\int_Q \theta^2\varphi Z^2dxdt.
\end{aligned}
\end{align}
By choosing a large enough $\lambda\geq C(T+T^2)$, one absorbs the third and fourth terms in the right-hand side of \eqref{2.01836} by using the first term, thus we obtain
\begin{align}
\begin{aligned}
\lambda\mu^2\mathbb{E}\int_Q \theta^2\varphi\vert\nabla z\vert^2dxdt\leq&\, C\Bigg[\lambda^3\mu^4\mathbb{E}\int_{Q} \theta^2\varphi^3z^2dxdt+\mathbb{E}\int_Q \theta^2 F_0^2dxdt\\
\label{2.019}&\quad\;
+\lambda\mu^2\mathbb{E}\int_Q \theta^2\varphi\vert F\vert^2dxdt+\lambda\mu^2\mathbb{E}\int_{Q} \theta^2\varphi Z^2dxdt\Bigg].
\end{aligned}
\end{align}
Finally, by combining \eqref{2.014} and \eqref{2.019} and taking a large $\lambda\geq CT^2$, we deduce the desired Carleman estimate \eqref{1.014}. This completes the proof of Theorem \ref{thm02.1}.
\end{proof}
\section{Global Carleman estimate for forward stochastic parabolic equations}\label{sec3}
In this section, we consider the following forward linear stochastic parabolic equation
	\begin{equation}\label{3.010}
		\begin{cases}
			\begin{array}{ll}
			dz-\nabla\cdot(A\nabla z) dt = (F_1+\nabla\cdot F)dt + F_2 dW(t)	 &\textnormal{in}\,\,Q,\\
				z=0 &\textnormal{on}\,\,\Sigma,\\
			 z(0)=z_0 &\textnormal{in}\,\,G,
			\end{array}
		\end{cases}
	\end{equation}
where $z_0\in L^2_{\mathcal{F}_0}(\Omega;L^2(G))$, $F_1,F_2\in L^2_\mathcal{F}(0,T;L^2(G))$, and $F\in L^2_\mathcal{F}(0,T;L^2(G;\mathbb{R}^N))$. 

Firstly, noting that if $F_2\equiv0$, equation \eqref{3.010} becomes a random parabolic equation, then inspiring from the well-known Carleman estimate for deterministic parabolic equations in \cite[Lemma 2.1]{fernandez2006global} and \eqref{2.2012}, we have the following Carleman estimate.
\begin{lm}\label{lm3.1}
There exist constants $C>0$, $\mu_0>1$, and $\lambda_0>1$ depending only on $G$, $G_0$ and $A$ such that for all $\mu\geq\mu_0$ and $\lambda\geq\lambda_0$, $F_1\in L^2_\mathcal{F}(0,T;L^2(G))$, $F\in L^2_\mathcal{F}(0,T;L^2(G;\mathbb{R}^N))$, and $z_0\in L^2_{\mathcal{F}_0}(\Omega;L^2(G))$, the corresponding solution of \eqref{3.010} $($with $F_2\equiv0$$)$ satisfies
\begin{align}
\begin{aligned}\label{3.20102}
&\,\lambda^3\mu^4\mathbb{E}\int_Q \theta^2\varphi^3z^2dxdt+\lambda\mu^2\mathbb{E}\int_Q \theta^2\varphi \vert\nabla z\vert^2dxdt\\
&\leq\, C\Bigg[\lambda^3\mu^4\mathbb{E}\int_{Q_0} \theta^2\varphi^3z^2dxdt+\mathbb{E}\int_Q\theta^2F_1^2dxdt+\lambda^2\mu^2\mathbb{E}\int_Q \theta^2\varphi^2 \vert F\vert^2 dxdt\Bigg].
\end{aligned}
\end{align}
\end{lm}

Throughout this section, we set $\mu=\mu_0$ and $\lambda\geq\lambda_0$, where $\mu_0$ and $\lambda_0$ are given in the above Lemma and we also choose the same weight functions $\theta$ and $\varphi$ as defined in \eqref{2.2012}. For any $F_1,F_2\in L_\mathcal{F}^2(0,T;L^2(G))$, $F\in L_\mathcal{F}^2(0,T;L^2(G;\mathbb{R}^N))$ and $z_0\in L^2_{\mathcal{F}_0}(\Omega;L^2(G))$, let $z$ be the corresponding solution to \eqref{3.010}. Let us now consider $(r,R)$ be the solution of the following backward stochastic parabolic equation with a control $u\in L_\mathcal{F}^2(0,T;L^2(G_0))$
\begin{equation}\label{3.6}
\begin{cases}
\begin{array}{ll}
dr + \nabla\cdot(A\nabla r) \,dt = (\lambda^3\theta^2\varphi^3z + \mathbbm{1}_{G_0}u) \,dt + R \,dW(t) &\textnormal{in}\,\,Q,\\
r=0 &\textnormal{on}\,\,\Sigma,\\
r(T)=0 &\textnormal{in}\,\,G.
\end{array}
\end{cases}
\end{equation}

Using Lemma \ref{lm3.1} and repeating some steps from \cite[Proposition 2.2]{liu2014global}, we have the following result.
\begin{prop}\label{prop3.1}
There exists a control $\hat{u}\in L^2_\mathcal{F}(0,T;L^2(G_0))$ such that the associated solution $(\hat{r},\hat{R})\in (L^2_\mathcal{F}(\Omega;C([0,T];L^2(G)))\bigcap L^2_\mathcal{F}(0,T;H^1_0(G)))\times L^2_\mathcal{F}(0,T;L^2(G))$ to equation \eqref{3.6} satisfies $\hat{r}(0,\cdot)=0$ in $G$, $\mathbb{P}\textnormal{-a.s.}$ Moreover, there exists a constant $C>0$ depending only on $G$, $G_0$, $\mu_0$ and $A$ so that
\begin{align}
\begin{aligned}
&\,\lambda^{-3} \mathbb{E}\int_{Q} \theta^{-2}\varphi^{-3}\hat{u}^2 dx dt +\mathbb{E}\int_Q \theta^{-2}\hat{r}^2 dx dt\\
&+\lambda^{-2}\mathbb{E}\int_Q \theta^{-2}\varphi^{-2}\vert\nabla\hat{r}\vert^2dx dt+\lambda^{-2}\mathbb{E}\int_Q \theta^{-2}\varphi^{-2}\hat{R}^2dx dt\\
\label{3.7}&
\leq C\,\lambda^3 \mathbb{E}\int_Q \theta^2\varphi^3 z^2 dxdt,
\end{aligned}
\end{align}
for all $\lambda\geq C(T+T^2)$.
\end{prop}

The main result of this section is the following global Carleman estimate for \eqref{3.010}, whose proof is based on Proposition \ref{prop3.1}.
\begin{thm}\label{thm1.1}
There exists a positive constant $C=C(G,G_0,\mu_0,A)$ such that for all  $F_1,F_2\in L^2_\mathcal{F}(0,T;L^2(G))$, $F\in L^2_\mathcal{F}(0,T;L^2(G;\mathbb{R}^N))$ and $z_0\in L^2_{\mathcal{F}_0}(\Omega;L^2(G))$, the weak solution $z$ of \eqref{3.010} satisfies that
\begin{align}
\begin{aligned}
    &\,\lambda^3\mathbb{E}\int_Q \theta^2\varphi^3z^2dxdt + \lambda\mathbb{E}\int_Q \theta^2\varphi\vert\nabla z\vert^2dxdt\\
&\leq C\Bigg[\lambda^3\mathbb{E}\int_{Q_0} \theta^2\varphi^3z^2dxdt+\mathbb{E}\int_Q \theta^2F_1^2dxdt\\
\label{3.202002}&\quad\quad\,\;+\lambda^2\mathbb{E}\int_Q \theta^2\varphi^2F_2^2dxdt+\lambda^2\mathbb{E}\int_Q \theta^2\varphi^2\vert F\vert^2dxdt\Bigg],
\end{aligned}
\end{align}
for all $\lambda\geq C(T+T^2)$.
\end{thm}
\begin{proof}
Let $(\hat{r},\hat{R})$ be the solution of equation \eqref{3.6} with control $\hat{u}$ obtained in Proposition \ref{prop3.1}. Using systems \eqref{3.010} and \eqref{3.6}, applying Itô's formula in Lemma \ref{lm1.1}, integrating the result on $(0,T)$ and taking the mean value on both sides, we get by Lemma \ref{lm1.2} that
\begin{equation}\label{3.9}
    \lambda^3\mathbb{E}\int_Q \theta^2\varphi^3z^2 dxdt = -\mathbb{E}\int_Q (\hat{r}F_1-\nabla\hat{r}\cdot F+\mathbbm{1}_{G_0}\hat{u}z+\hat{R}F_2)dxdt.
\end{equation}
By Young's inequality, (\ref{3.9}) implies that for all $\rho>0$
\begin{align}
\begin{aligned}
&\,\lambda^3\mathbb{E}\int_Q \theta^2\varphi^3z^2 dxdt\\
&\leq \rho\Bigg[\mathbb{E}\int_Q \theta^{-2}\hat{r}^2 dxdt+\lambda^{-2}\mathbb{E}\int_Q \theta^{-2}\varphi^{-2}\vert\nabla\hat{r}\vert^2 dxdt\\
&\qquad\,+\lambda^{-3}\mathbb{E}\int_{Q} \theta^{-2}\varphi^{-3} \hat{u}^2 dxdt+\lambda^{-2}\mathbb{E}\int_Q \theta^{-2}\varphi^{-2} \hat{R}^2 dxdt\Bigg]\\
\label{2.550}&\quad+C(\rho)\Bigg[\mathbb{E}\int_Q \theta^2F_1^2 dxdt+\lambda^2\mathbb{E}\int_Q \theta^2\varphi^2\vert F\vert^2 dxdt\\
&\qquad\qquad\;\;+\lambda^3\mathbb{E}\int_{Q_0} \theta^2\varphi^3z^2 dxdt+\lambda^2\mathbb{E}\int_Q \theta^2\varphi^2 F_2^2 dxdt\Bigg].
\end{aligned}
\end{align}
Using \eqref{3.7} and choosing a small $\rho>0$, the inequality \eqref{2.550} implies that
\begin{align}
\begin{aligned}
\lambda^3\mathbb{E}\int_Q \theta^2\varphi^3z^2 dxdt\leq&\, C\Bigg[\lambda^3\mathbb{E}\int_{Q_0} \theta^2\varphi^3z^2 dxdt+\mathbb{E}\int_Q \theta^2F_1^2 dxdt\\
\label{2.5}&\quad\,
+\lambda^2\mathbb{E}\int_Q \theta^2\varphi^2\vert F\vert^2 dxdt+\lambda^2\mathbb{E}\int_Q \theta^2\varphi^2 F_2^2 dxdt\Bigg],
\end{aligned}
\end{align}
for $\lambda\geq C(T+T^2)$.

On the other hand, using Itô's formula, we compute $d(\theta^2\varphi z,z)_{L^2(G)}$, integrating the equality on $(0,T)$ and taking the expectation on both sides, we get
\begin{align}
\begin{aligned}
2\mathbb{E}\int_Q \theta^2\varphi A\nabla z\cdot\nabla zd xdt=&\,\mathbb{E}\int_Q (\theta^2\varphi)_tz^2dxdt-2\mathbb{E}\int_Q zA\nabla z\cdot\nabla(\theta^2\varphi) dxdt\\
\label{2.6}& -2\mathbb{E}\int_Q zF\cdot\nabla(\theta^2\varphi)dxdt-2\mathbb{E}\int_Q \theta^2\varphi F\cdot\nabla z dxdt\\
&+2\mathbb{E}\int_Q \theta^2\varphi zF_1dxdt+\mathbb{E}\int_Q \theta^2\varphi F_2^2dxdt.
\end{aligned}
\end{align}
We have for $\lambda\geq CT^2$
\begin{equation}\label{2.7}
    \vert (\theta^2\varphi)_t\vert\leq CT\lambda\theta^2\varphi^3,\;\qquad\quad \vert\nabla(\theta^2\varphi)\vert\leq C\lambda\theta^2\varphi^2.
\end{equation}
From \eqref{2.6}, \eqref{2.7} and the assumptions on $A$, we get
\begin{align}
\begin{aligned}
&\,\mathbb{E}\int_Q \theta^2\varphi\vert\nabla z\vert^2dxdt\\
&\leq CT\lambda \mathbb{E}\int_Q \theta^2\varphi^3z^2dxdt+C\lambda\mathbb{E}\int_Q \theta^2\varphi^2\vert z\vert\vert\nabla z\vert dxdt\\
\label{2.8}&
\quad+C\lambda\mathbb{E}\int_Q \theta^2\varphi^2\vert F\vert\vert z\vert dxdt+C\mathbb{E}\int_Q \theta^2\varphi\vert F\vert\vert\nabla z\vert dxdt\\
& \quad+C\mathbb{E}\int_Q \theta^2\varphi\vert F_1\vert\vert z\vert dxdt+C\mathbb{E}\int_Q \theta^2\varphi F_2^2dxdt.
\end{aligned}
\end{align}
Applying Young's inequality on the right-hand side of \eqref{2.8}, absorbing the gradient terms by using the left-hand side, we arrive at 
\begin{align}
\begin{aligned}
\mathbb{E}\int_Q \theta^2\varphi\vert\nabla z\vert^2dxdt\leq&\, C\Bigg[T\lambda \mathbb{E}\int_Q \theta^2\varphi^3z^2dxdt+\lambda^2\mathbb{E}\int_Q \theta^2\varphi^3z^2dxdt\\
\label{2.9}&\quad\;
+\mathbb{E}\int_Q \theta^2\varphi\vert F\vert^2dxdt+\lambda^{-1}\mathbb{E}\int_Q \theta^2 F_1^2dxdt\\
&\quad\;+\lambda\mathbb{E}\int_Q \theta^2\varphi^2z^2dxdt+\mathbb{E}\int_Q \theta^2\varphi F_2^2dxdt\Bigg].
\end{aligned}
\end{align}
Multiplying \eqref{2.9} by $\lambda$, combining the new inequality with \eqref{2.5} and taking a large enough $\lambda\geq C(T+T^2)$, we get \eqref{3.202002}. This concludes the proof of Theorem \ref{thm1.1}.
\end{proof}
\begin{rmk}
In the above proof, we may notice that $A$ needs only to be $W^{1,\infty}$ in space.
\end{rmk}
\section{Observability problems}\label{sec4}
This section is devoted to establishing the observability inequalities of the adjoint equations \eqref{1.3} and \eqref{1.5}. Firstly, let us prove the observability inequality \eqref{1.04}.
\begin{thm}\label{thm4.1}
For all $z_T\in L^2_{\mathcal{F}_T}(\Omega;L^2(G))$, the corresponding solution $(z,Z)$ to equation \eqref{1.3} satisfies the observability inequality \eqref{1.04}.
\end{thm}
\begin{proof}
By applying \eqref{1.014} with $\mu=C$, $F_0=-a_1z-a_2Z\;\textnormal{and}\; F=zB_1+ZB_2$, we deduce that for $\lambda\geq C(T+T^2)$
\begin{align}
\begin{aligned}
&\,\lambda^3\mathbb{E}\int_Q \theta^2\varphi^3z^2dxdt\\
&\leq  C\lambda^3\mathbb{E}\int_{Q_0} \theta^2\varphi^3z^2dxdt+C\mathbb{E}\int_Q \theta^2 \vert a_1z+a_2Z\vert^2dxdt\\
\label{1.0142}&\quad +C\lambda^2\mathbb{E}\int_Q \theta^2\varphi^2\vert zB_1+ZB_2\vert^2dxdt+C\lambda^2\mathbb{E}\int_Q \theta^2\varphi^2Z^2dxdt.
\end{aligned}
\end{align}
Now, it is sufficient to take $\lambda\geq CT^2(\vert a_1\vert_\infty^{2/3}+\vert a_2\vert_\infty^{2/3})$ to get
\begin{equation}\label{4.20120}
    C\mathbb{E}\int_Q \theta^2 \vert a_1z+a_2Z\vert^2dxdt \leq \frac{1}{4}\lambda^3\mathbb{E}\int_Q \theta^2\varphi^3z^2dxdt +  C\lambda^{3}\mathbb{E}\int_Q \theta^2\varphi^3Z^2dxdt,
\end{equation}
and by $\lambda\geq CT^2(\vert B_1\vert_\infty^2+\vert B_2\vert_\infty^2)$, we also have
\begin{align}\label{4.50120}
\begin{aligned}
&\,C\lambda^2\mathbb{E}\int_Q \theta^2\varphi^2\vert zB_1+ZB_2\vert^2dxdt \\
    &\leq \frac{1}{4}\lambda^3\mathbb{E}\int_Q \theta^2\varphi^3z^2dxdt +  C\lambda^{3}\mathbb{E}\int_Q \theta^2\varphi^3Z^2dxdt.
\end{aligned}
\end{align}
Combining \eqref{1.0142}, \eqref{4.20120} and \eqref{4.50120}, we conclude that
\begin{equation*}
    \mathbb{E}\int_Q \theta^2 \varphi^3z^2 dxdt \leq C\mathbb{E}\int_{Q_0} \theta^2\varphi^3z^2dxdt +  C\mathbb{E}\int_Q \theta^2\varphi^3Z^2dxdt,
\end{equation*}
for $\lambda\geq \lambda_1=C\Big[T+T^2\big(1+\vert a_1\vert_\infty^{2/3}+\vert a_2\vert_\infty^{2/3}+\vert B_1\vert_\infty^2+\vert B_2\vert_\infty^2\big)\Big]$. Then, it follows that
\begin{equation}\label{2.3701}
\mathbb{E}\int_{T/4}^{3T/4}\int_G \theta^2\varphi^3z^2dxdt\leq  C\Bigg[\mathbb{E}\int_{Q_0} \theta^2\varphi^3z^2dxdt+\mathbb{E}\int_Q \theta^2\varphi^3Z^2dxdt\Bigg],
\end{equation}
for $\lambda\geq \lambda_1$. It is not difficult to see that the function
$$t\mapsto\min_{x\in\overline{G}}\,[\theta^2(t,x)\varphi^3(t,x)]$$
reaches its minimum in $[T/4,3T/4]$ at $t=T/4$ and also 
$$t\mapsto\max_{x\in\overline{G}}\,[\theta^2(t,x)\varphi^3(t,x)]$$
reaches its maximum in $(0,T)$ at $t=T/2$.

Now, fixing $\lambda=\lambda_1$, it is easy to check that
\begin{equation}\label{3.00116}
\max_{x\in\overline{G}}\theta^2\Big(\frac{T}{2},x\Big)\,\cdot\,\min_{x\in\overline{G}}\theta^{-2}\Big(\frac{T}{4},x\Big)=e^{C\lambda_1T^{-2}},\,
\end{equation}
and
\begin{equation}\label{3.00117}
\max_{x\in\overline{G}}\varphi^3\Big(\frac{T}{2},x\Big)\,\cdot\,\min_{x\in\overline{G}}\varphi^{-3}\Big(\frac{T}{4},x\Big)= C.
\end{equation}
Combining \eqref{2.3701}, \eqref{3.00116} and \eqref{3.00117}, we deduce that
 \begin{equation}\label{2.41012}
     \mathbb{E}\int_{T/4}^{3T/4}\int_G z^2 dxdt \leq e^{C r_1} \,\Bigg[ \mathbb{E}\int_{Q_0} z^2 dxdt+\mathbb{E}\int_{Q} Z^2 dxdt\Bigg],
 \end{equation}
 where $r_1=1+\frac{1}{T}+\vert a_1\vert_\infty^{2/3}+\vert a_2\vert_\infty^{2/3}+\vert B_1\vert_\infty^2+\vert B_2\vert_\infty^2$.
 
On the other hand, let $t\in[0,T]$, by Itô's formula in Lemma \ref{lm1.1}, computing $d_s(z,z)_{L^2(G)}$, integrating the result w.r.t $s\in[0,t]$ and taking the mean value on both sides, we obtain
\begin{align}
\begin{aligned}
&\,\mathbb{E}\int_G \vert z(0)\vert^2dx-\mathbb{E}\int_G \vert z(t)\vert^2dx\\
&\leq -2\beta\mathbb{E}\int_0^t\int_G \vert\nabla z\vert^2dxds+C\vert a_1\vert_\infty\mathbb{E}\int_0^t\int_G z^2dxds\\
&\quad\,+C\vert a_2\vert_\infty\mathbb{E}\int_0^t\int_G \vert z\vert\vert Z\vert dxds-\mathbb{E}\int_0^t\int_G Z^2 dxds\\
\label{2.42012}&\quad\,+C\vert B_1\vert_\infty\mathbb{E}\int_0^t\int_G \vert z\vert\vert\nabla z\vert dxds+C\vert B_2\vert_\infty\mathbb{E}\int_0^t\int_G \vert Z\vert\vert\nabla z\vert dxds.
\end{aligned}
\end{align}
By applying Young's inequality on the right-hand side of \eqref{2.42012}, one absorbs gradient terms by using the first term, and thus we arrive at
\begin{align}\label{2.43012}
\begin{aligned}
\mathbb{E}\int_G \vert z(0)\vert^2dx &\leq \mathbb{E}\int_G \vert z(t)\vert^2dx+Cr_2\mathbb{E}\int_0^t\int_G z^2dxds+C\vert B_2\vert_\infty^2\mathbb{E}\int_Q Z^2 dxdt,
\end{aligned}
\end{align}
with $r_2=\vert a_1\vert_\infty+\vert a_2\vert^2_\infty+\vert B_1\vert^2_\infty$. Hence, by Gronwall inequality, \eqref{2.43012} implies that
\begin{equation}\label{2.45012}
\mathbb{E}\int_G \vert z(0)\vert^2dx \leq e^{Cr_2T}\mathbb{E}\int_G \vert z(t)\vert^2dx+C\vert B_2\vert_\infty^2\mathbb{E}\int_Q Z^2 dxdt.
\end{equation}
Integrating \eqref{2.45012} on $(T/4,3T/4)$ and combining the obtained inequality with \eqref{2.41012}, we end up with 
\begin{equation}\label{4.130012}
    \mathbb{E}\int_G \vert z(0)\vert^2dx \leq e^{C\big(\frac{1}{T}+r_1+r_2T+\vert B_2\vert_\infty^2\big)}\Bigg[\mathbb{E}\int_{Q_0} z^2dxdt+\mathbb{E}\int_Q Z^2dxdt\Bigg].
\end{equation}
Finally, from estimate \eqref{4.130012}, it is straightforward to get the observability inequality \eqref{1.04}. This concludes the proof of Theorem \ref{thm4.1}.
\end{proof}

Now, we establish the observability inequality \eqref{1.8}.
\begin{thm}\label{thm4.2}
For all initial state $z_0\in L^2_{\mathcal{F}_0}(\Omega;L^2(G))$, the associated solution $z$ of \eqref{1.5} satisfies the observability inequality \eqref{1.8}.
\end{thm}
\begin{proof}
Using the Carleman estimate \eqref{3.202002}, and choosing $\mu=\mu_0$, $F_1=-a_1z$, $F_2=-a_2z$ and $F=zB$, then we deduce that for all $\lambda\geq C(T+T^2)$
\begin{align}
\begin{aligned}
\lambda^3\mathbb{E}\int_Q \theta^2\varphi^3z^2dxdt 
\leq& C\Bigg[\lambda^3\mathbb{E}\int_{Q_0} \theta^2\varphi^3z^2dxdt+\mathbb{E}\int_Q \theta^2\vert a_1z\vert^2dxdt\\
\label{3.13}& \quad\,+\lambda^2\mathbb{E}\int_Q \theta^2\varphi^2\vert a_2z\vert^2dxdt+\lambda^2\mathbb{E}\int_Q \theta^2\varphi^2\vert zB\vert^2dxdt\Bigg].
\end{aligned}
\end{align}
By choosing a large enough $\lambda$ in \eqref{3.13}, we use the left-hand side to absorb the second, third, and fourth terms on the right-hand side. Here, it is sufficient to take $\lambda\geq CT^2\big(\vert a_1\vert^{2/3}_\infty+\vert a_2\vert^{2}_\infty+\vert B\vert^{2}_\infty\big)$ to obtain
\begin{align}
\begin{aligned}
&\,C\mathbb{E}\int_Q \theta^2\vert a_1z\vert^2dxdt+C\lambda^2\mathbb{E}\int_Q \theta^2\varphi^2\vert a_2z\vert^2dxdt
\\
\label{3.14}&+C\lambda^2\mathbb{E}\int_Q \theta^2\varphi^2\vert zB\vert^2dxdt\leq\frac{1}{2}\lambda^3\mathbb{E}\int_Q \theta^2\varphi^3z^2dxdt.
\end{aligned}
\end{align}
Combining \eqref{3.13} and \eqref{3.14}, we get for $\lambda\geq\overline{\lambda}_1= C\big[T+T^2\big(1+\vert a_1\vert^{2/3}_\infty+\vert a_2\vert^{2}_\infty+\vert B\vert^{2}_\infty\big)\big]$
\begin{equation}\label{3.15}
\mathbb{E}\int_Q \theta^2\varphi^3z^2dxdt 
 \leq C\,\mathbb{E}\int_{Q_0} \theta^2\varphi^3z^2dxdt.
 \end{equation}
Now, set $\lambda=\overline{\lambda}_1$, and combining \eqref{3.15}, \eqref{3.00116} and \eqref{3.00117}, we conclude that
 \begin{equation}\label{3.18}
\mathbb{E}\int_{T/4}^{3T/4}\int_G z^2(t,x) dxdt \leq e^{C\big(1+\frac{1}{T}+\vert a_1\vert_\infty^{2/3}+\vert a_2\vert_\infty^{2}+\vert B\vert_\infty^2\big)} \, \mathbb{E}\int_{Q_0} z^2(t,x) dxdt.
 \end{equation}
 
Let $t\in(0,T)$, computing $d_s(z,z)_{L^2(G)}$ by using Itô's formula and integrating the result w.r.t $s\in(t,T)$, we obtain
\begin{align}
\begin{aligned}
&\,\mathbb{E}\int_G  z^2(T,x) dx-\mathbb{E}\int_G  z^2(t,x) dx\\
&= -2\mathbb{E}\int_t^T\int_G A\nabla z\cdot\nabla z dxds-2\mathbb{E}\int_t^T\int_G a_1z^2dxds\\
\label{4.23012}&\,\quad-2\mathbb{E}\int_t^T\int_G zB\cdot\nabla z dxds+\mathbb{E}\int_t^T\int_G  a_2^2 z^2dxds.
\end{aligned}
\end{align}
By the assumptions on $A$ and Young's inequality, \eqref{4.23012} implies that
\begin{align}
\begin{aligned}
&\,\mathbb{E}\int_G  z^2(T,x) dx-\mathbb{E}\int_G  z^2(t,x) dx\\
&\leq-2\beta\mathbb{E}\int_t^T\int_G \vert\nabla z\vert^2 dxds+2\vert a_1\vert_\infty\mathbb{E}\int_t^T\int_G z^2dxds\\
&\,\quad+\rho\mathbb{E}\int_t^T\int_G \vert\nabla z\vert^2 dxds+C(\rho)\vert B\vert_\infty^2\mathbb{E}\int_t^T\int_G z^2 dxds\\
\label{3.19}&\,\quad+\vert a_2\vert^2_\infty\mathbb{E}\int_t^T\int_G z^2dxds.
\end{aligned}
\end{align}
By choosing a small enough $\rho$ in \eqref{3.19}, it is easy to see
\begin{equation}\label{4.24032}
\mathbb{E}\int_G  z^2(T,x) dx\leq\mathbb{E}\int_G  z^2(t,x) dx+C(\vert a_1\vert_\infty+\vert B\vert_\infty^2+\vert a_2\vert^2_\infty)\mathbb{E}\int_t^T\int_G z^2dxds.
\end{equation}
Hence, by Gronwall's inequality, \eqref{4.24032} provides that 
\begin{equation}\label{3.20}
\mathbb{E}\int_G  z^2(T,x) dx\leq e^{CT(\vert a_1\vert_\infty+\vert B\vert_\infty^2+\vert a_2\vert^2_\infty)}\mathbb{E}\int_G  z^2(t,x) dx.
\end{equation}
Finally, by combining \eqref{3.18} and \eqref{3.20}, we conclude the proof of Theorem \ref{thm4.2}.
\end{proof}

\section{Proof of Theorems \ref{thm01.1} and \ref{thm1.2}}\label{sec5}
Only Theorem \ref{thm01.1} is proved here, the proof of Theorem \ref{thm1.2} can be provided similarly.
\begin{proof}[Proof of Theorem \ref{thm01.1}]
Let us fix $\varepsilon>0$, $y_0\in L^2_{\mathcal{F}_0}(\Omega;L^2(G))$, and then consider the following optimal control problem
\begin{align}\label{05.1}
\textnormal{inf}\big\{J_\varepsilon(u,v)\,\vert\, (u,v)\in L^2_\mathcal{F}(0,T;L^2(G))\times L^2_\mathcal{F}(0,T;L^2(G))\big\},
\end{align}
with
$$J_\varepsilon(u,v)=\frac{1}{2}\mathbb{E}\int_Q u^2 dxdt+\frac{1}{2}\mathbb{E}\int_Q v^2 dxdt+\frac{1}{2\varepsilon}\mathbb{E}\int_G \vert y(T,\cdot)\vert^2dx,$$
where $y$ is the solution of \eqref{1.1} with initial state $y_0$ and controls $u$ and $v$. It is easy to check that $J_\varepsilon$ is well defined, continuous, strictly convex, and coercive. Then, the problem \eqref{05.1} has a unique solution $(u_\varepsilon,v_\varepsilon)\in L^2_\mathcal{F}(0,T;L^2(G))\times L^2_\mathcal{F}(0,T;L^2(G))$.

By the standard duality arguments, the minimum $(u_\varepsilon,v_\varepsilon)$ is characterized by 
\begin{align}\label{05.2}
    u_\varepsilon=\mathbbm{1}_{G_0}z_\varepsilon\,,\qquad v_\varepsilon=Z_\varepsilon,
\end{align}
with $(z_\varepsilon,Z_\varepsilon)$ is the solution of the following backward stochastic parabolic equation
\begin{equation*}
\begin{cases}
\begin{array}{ll}
dz_\varepsilon + \nabla\cdot(A\nabla z_\varepsilon) \,dt = (-a_1 z_\varepsilon-a_2Z_\varepsilon+\nabla\cdot(z_\varepsilon B_1+Z_\varepsilon B_2)) \,dt + Z_\varepsilon \,dW(t) &\textnormal{in}\,\,Q,\\
z_\varepsilon=0 &\textnormal{on}\,\,\Sigma,\\
z_\varepsilon(T)=-\frac{1}{\varepsilon}y_\varepsilon(T,\cdot) &\textnormal{in}\,\,G,
			\end{array}
		\end{cases}
\end{equation*}
where $y_\varepsilon$ is the solution of \eqref{1.1} with controls $u_\varepsilon$ and $v_\varepsilon$.

By Itô's formula, we have that
\begin{align}\label{05.3}
d(y_\varepsilon z_\varepsilon)=\mathbbm{1}_{G_0}u_\varepsilon z_\varepsilon+v_\varepsilon Z_\varepsilon.
\end{align}
Integrating \eqref{05.3} on $Q$ and taking the expectation on both sides, we get
\begin{align}\label{05.4}
\mathbb{E}\int_{Q_0}z_\varepsilon^2dxdt+\mathbb{E}\int_{Q}Z_\varepsilon^2dxdt+\frac{1}{\varepsilon}\mathbb{E}\int_G \vert y_\varepsilon(T,\cdot)\vert^2dx=\mathbb{E}\int_G y_0 z_\varepsilon(0,\cdot)dx
\end{align}
By Young's inequality, \eqref{05.4} implies that
\begin{align}\label{05.5}
\begin{aligned}
&\,\mathbb{E}\int_{Q_0}z_\varepsilon^2dxdt+\mathbb{E}\int_{Q}Z_\varepsilon^2dxdt+\frac{1}{\varepsilon}\mathbb{E}\int_G \vert y_\varepsilon(T,\cdot)\vert^2dx\\
&\leq \frac{e^{CK}}{2}\vert y_0\vert^2_{L^2_{\mathcal{F}_0}(\Omega;L^2(G))}+\frac{1}{2e^{CK}}\vert z_\varepsilon(0,\cdot)\vert^2_{L^2_{\mathcal{F}_0}(\Omega;L^2(G))},
\end{aligned}
\end{align}
where $e^{CK}$ is the same constant appearing in \eqref{1.04}.

Now, combining \eqref{05.5}, \eqref{1.04} and \eqref{05.2}, it is easy to see that
\begin{align}\label{05.6}
\vert u_\varepsilon\vert^2_{L^2_\mathcal{F}(0,T;L^2(G))}+\vert v_\varepsilon\vert^2_{L^2_\mathcal{F}(0,T;L^2(G))}+\frac{2}{\varepsilon}\mathbb{E}\int_G \vert y_\varepsilon(T,\cdot)\vert^2dx\leq e^{CK}\vert y_0\vert^2_{L^2_{\mathcal{F}_0}(\Omega;L^2(G))}.
\end{align}
It follows from \eqref{05.6} that there exist sub-sequences of $(u_\varepsilon)$ and $(v_\varepsilon)$ which still denoted also respectively by $(u_\varepsilon)$ and $(v_\varepsilon)$ such that as $\varepsilon\rightarrow0$, we have
\begin{align}\label{05.7}
\begin{aligned}
&\,u_\varepsilon \longrightarrow \hat{u},\,\,\,\textnormal{weakly in}\,\,\,L^2((0,T)\times\Omega;L^2(G));
\\
&v_\varepsilon \longrightarrow \hat{v},\,\,\,\textnormal{weakly in}\,\,\,L^2((0,T)\times\Omega;L^2(G)).
\end{aligned}
\end{align}
Noting also that $\textnormal{Supp}\, \hat{u}\subset [0,T]\times \overline{G_0}$. Then, finally, if $\hat{y}$ denotes the solution of \eqref{1.1} associated to $y_0$, $\hat{u}$ and $\hat{v}$, then according to \eqref{05.7} and \eqref{05.6}, we have that
$$\vert \hat{u}\vert^2_{L^2_\mathcal{F}(0,T;L^2(G))}+\vert \hat{v}\vert^2_{ L^2_\mathcal{F}(0,T;L^2(G))}\leq e^{CK}\,\vert y_0\vert^2_{L^2_{\mathcal{F}_0}(\Omega;L^2(G))},$$
and $\hat{y}(T,\cdot)=0$ in $G$, $\mathbb{P}\textnormal{-a.s.}$ This completes the proof of Theorem \ref{thm01.1}.
\end{proof}
\section*{Appendix. Proof of Lemma \ref{lm2.2}}
Here, we follow the same strategy of computation as in the proof of \cite[Theorem 9.39]{lu2021mathematical}. In our case, we need to treat carefully the dependence of Carleman's parameters w.r.t the final control time $T$. First, we recall that the diffusion matrix operator $A=(a^{jk})_{j,k=1,2,...,N}$ satisfies the following assumptions
 \begin{enumerate}
\item $a^{jk}\in L^\infty_\mathcal{F}(\Omega;C^1([0,T];W^{2,\infty}(G)))$, and $a^{jk}=a^{kj}$\,\, for all $j,k=1,2,...,N$.
\item There exists a constant $\beta>0$ such that
    $$\sum_{j,k=1}^N a^{jk}\xi_j\xi_k \geq \beta\vert\xi\vert^2,\qquad (\omega,t,x,\xi)\equiv(\omega,t,x,\xi_1,...,\xi_N)\in\Omega\times Q\times\mathbb{R}^N,$$
\end{enumerate}
and set the two functions $l\in C^{1,3}(Q)$ and $\Psi\in C^{1,2}(Q)$. For simplicity of notations, we denote by $y_{x_j}=\frac{\partial y}{\partial x_j}$ the partial derivative of a function $y$ with respect to $x_j$, where $x_j$ is the $j$-component of a generic point $x=(x_1,...,x_N)\in\mathbb{R}^N$. Similarly, we also denote by $y_{x_jx_k}$ the second partial derivative of $y$ w.r.t variables $x_j$ and $x_k$, and so on. Put
\begin{equation}\label{5.1111012}
    \begin{cases}
        \mathcal{A}= \displaystyle\sum_{j,k=1}^N (a^{jk}l_{x_j}l_{x_k}-a^{jk}_{x_k}l_{x_j}-a^{jk}l_{x_jx_k})-\Psi-l_t,\\
        \mathcal{B}= 2\Bigg[\mathcal{A}\Psi+\displaystyle\sum_{j,k=1}^N (\mathcal{A}a^{jk}l_{x_j})_{x_k}\Bigg]-\mathcal{A}_t+\sum_{j,k=1}^N (a^{jk}\Psi_{x_k})_{x_j},\\
        c^{jk}=\displaystyle\sum_{j',k'=1}^N\Big[2a^{jk'}(a^{j'k}l_{x_{j'}})_{x_{k'}}-(a^{jk}a^{j'k'}l_{x_{j'}})_{x_{k'}}\Big]+\frac{a^{jk}_t}{2}-\Psi a^{jk}.
    \end{cases}
\end{equation}

From \cite[Theorem 9.26]{lu2021mathematical}, we have the following fundamental weighted identity for the stochastic parabolic-like operator \textit{``$dz+\nabla\cdot(A\nabla z)dt$"}.
\begin{thm}\label{thm5.1}
Let $z$ be an $H^2(G)$-valued Itô process. Set $\theta=e^l$ and $h=\theta z$. Then, for any $t\in[0,T]$ and a.e. $(x,\omega)\in G\times\Omega$,
\begin{align}
\begin{aligned}
&\,2\theta[\nabla\cdot(A\nabla h)+\mathcal{A}h][dz+\nabla\cdot(A\nabla z)dt]-2\nabla\cdot(A\nabla h \,dh)\\
&+2\sum_{j,k=1}^N\Bigg[\sum_{j',k'=1}^N \big(2a^{jk}a^{j'k'}l_{x_{j'}}h_{x_{j}}h_{x_{k'}}-a^{jk}a^{j'k'}l_{x_j}h_{x_{j'}}h_{x_{k'}}\big)\\
&\qquad\qquad\qquad\quad\;\;\;-\Psi a^{jk}h_{x_j}h+a^{jk}\Big(\mathcal{A}l_{x_j}+\frac{\Psi_{x_j}}{2}\Big)h^2\Bigg]_{x_k}dt\\
&=2\sum_{j,k=1}^N c^{jk}h_{x_j}h_{x_k}dt+\mathcal{B}h^2dt+d\Bigg(-\sum_{j,k=1}^N a^{jk}h_{x_j}h_{x_k}+\mathcal{A}h^2\Bigg)\\
\label{5.121301}& \;\;\;\,+2\big[\nabla\cdot(A\nabla h)+\mathcal{A}h\big]^2dt+\theta^2\sum_{j,k=1}^N a^{jk}(dz_{xj}+l_{x_j}dz)(dz_{x_k}+l_{x_k}dz)\\
& \;\;\;\,-\theta^2\mathcal{A}(dz)^2.
\end{aligned}
\end{align}
\end{thm}

In what follows, we choose $\Psi=-2\sum_{j,k=1}^N a^{jk}l_{x_jx_k}$ and $l$ is the function defined in \eqref{2.2012}. For a positive integer $n$, we denote by $O(\mu^n)$ a function of order $\mu^n$ for a large $\mu$, and likewise, we use the notation $O(e^{\mu\vert\psi\vert_\infty})$. It is easy to see that for any parameters $\lambda,\mu\geq1$, we have for any $j,k=1,2,...,N$
\begin{equation}\label{5.1201}
 l_t=\lambda\alpha_t,\quad l_{x_j}=\lambda\mu\varphi\psi_{x_j},\quad l_{x_jx_k}=\lambda\mu^2\varphi\psi_{x_j}\psi_{x_k}+\lambda\mu\varphi\psi_{x_jx_k},
 \end{equation}
and also
\begin{align}\label{5.1301}
\begin{aligned}
&\alpha_t=T\varphi^2 O(e^{2\mu\vert\psi\vert_\infty}),\quad \varphi_t=T\varphi^2O(1),\\
&\alpha_{tt}=T^2\varphi^3 O(e^{2\mu\vert\psi\vert_\infty}),\quad \varphi_{tt}=T^2\varphi^3O(1).
\end{aligned}
\end{align}

We have the following estimates on $\mathcal{A}$, $\mathcal{B}$ and $c^{jk}$ defined in \eqref{5.1111012}.
\begin{thm}\label{thm5.2}
    There exists $C=C(G,A)>0$ such that for large enough $\mu\geq C$, and $\lambda\geq C(T+T^2)$, we have
$$\mathcal{A}=\lambda^2\mu^2\varphi^2\sum_{j,k=1}^N a^{jk}\psi_{x_j}\psi_{x_k}+\lambda\varphi O(\mu^2)+\lambda T\varphi^2 O(e^{2\mu\vert\psi\vert_\infty}).$$
\begin{align*}
\mathcal{B}\geq&\, 2\beta^2\lambda^3\mu^4\varphi^3\vert\nabla\psi\vert^4+\lambda^3\varphi^3O(\mu^3)+\lambda^2\varphi^2O(\mu^4)+\lambda\varphi O(\mu^4)+\lambda^2T\varphi^3O(\mu^2e^{2\mu\vert\psi\vert_\infty}).
\end{align*}
$$\sum_{j,k=1}^N c^{jk}\xi_{j}\xi_{k}\geq [\beta^2\lambda\mu^2\varphi\vert\nabla\psi\vert^2+\lambda\varphi O(\mu)]\vert\xi\vert^2,\qquad\quad\forall\xi=(\xi_1,...,\xi_N)\in\mathbb{R}^N,$$
for all $x\in G$ with $\vert\nabla\psi(x)\vert>0$, and any $t\in [0,T]$, $\mathbb{P}\textnormal{-a.s.}$
\end{thm} 
\begin{proof}
For the estimate of $\mathcal{A}$: Noting that, we have
$$\mathcal{A}= \displaystyle\sum_{j,k=1}^N (a^{jk}l_{x_j}l_{x_k}-a^{jk}_{x_k}l_{x_j}+a^{jk}l_{x_jx_k})-l_t.$$
From \eqref{5.1201} and \eqref{5.1301}, it easy to see
\begin{align*}
\mathcal{A}=&\,\sum_{j,k=1}^N \big(\lambda^2\mu^2\varphi^2a^{jk}\psi_{x_j}\psi_{x_k}-\lambda\mu\varphi a^{jk}_{x_k}\psi_{x_j}+\lambda\mu^2\varphi a^{jk}\psi_{x_j}\psi_{x_k}+\lambda\mu\varphi a^{jk}\psi_{x_jx_k}\big)\\
&\,+\lambda T\varphi^2O(e^{2\mu\vert\psi\vert_\infty}).
\end{align*}
Then, we get for a large enough $\mu\geq C$ that
$$\mathcal{A}=\lambda^2\mu^2\varphi^2\sum_{j,k=1}^N a^{jk}\psi_{x_j}\psi_{x_k}+\lambda\varphi O(\mu^2)+\lambda T\varphi^2 O(e^{2\mu\vert\psi\vert_\infty}).$$

For the estimate of $\mathcal{B}$: We first have
$$\Psi=-2\sum_{j,k=1}^N a^{jk}l_{x_jx_k}=-2\sum_{j,k=1}^N a^{jk}(\lambda\mu^2\varphi\psi_{x_j}\psi_{x_k}+\lambda\mu\varphi\psi_{x_jx_k}).$$
Then, it is easy to see 
$$\Psi=-2\lambda\mu^2\varphi\sum_{j,k=1}^N a^{jk}\psi_{x_j}\psi_{x_k}+\lambda\varphi O(\mu),$$
and for all $j,k=1,2,...,N$
$$\Psi_{x_k}=-2\lambda\mu^3\varphi\sum_{j',k'=1}^N a^{j'k'}\psi_{x_{j'}}\psi_{x_{k'}}\psi_{x_k}+\lambda\varphi O(\mu^2),$$
$$\Psi_{x_jx_k}=-2\lambda\mu^4\varphi\sum_{j',k'=1}^Na^{j'k'}\psi_{x_{j'}}\psi_{x_{k'}}\psi_{x_j}\psi_{x_k}+\lambda\varphi O(\mu^3).$$
It follows that
\begin{align}
\begin{aligned}
\sum_{j,k=1}^N(a^{jk}\Psi_{x_k})_{x_j}&\,=\sum_{j,k=1}^N (a^{jk}_{x_j}\Psi_{x_k}+a^{jk}\Psi_{x_jx_k})\\
\label{5.141401}&\,\,=-2\lambda\mu^4\varphi\Bigg(\sum_{j,k=1}^N a^{jk}\psi_{x_j}\psi_{x_k}\Bigg)^2+\lambda\varphi O(\mu^3).
\end{aligned}
\end{align}
Then, for a large $\mu\geq C$, we obtain that
\begin{align}
\begin{aligned}
\mathcal{A}\Psi=&\,-2\lambda^3\mu^4\varphi^3\Bigg(\sum_{j,k=1}^N a^{jk}\psi_{x_j}\psi_{x_k}\Bigg)^2+\lambda^3\varphi^3O(\mu^3)+\lambda^2\varphi^2O(\mu^4)\\
\label{5.151501}&\,\,+\lambda^2T\varphi^3O(\mu^2e^{2\mu\vert\psi\vert_\infty}).
\end{aligned}
\end{align}
On the other hand, we have
\begin{align}\label{5.16170123}
\sum_{j,k=1}^N (\mathcal{A}a^{jk}l_{x_j})_{x_k}=\sum_{j,k=1}^N \big[\mathcal{A}_{x_k}a^{jk}l_{x_j}+\mathcal{A}\big(a^{jk}_{x_k}l_{x_j}+a^{jk}l_{x_jx_k}\big)\big].
\end{align}
It is easy to check for a large $\mu\geq C$ that
\begin{align}
\begin{aligned}
\sum_{j,k=1}^N \mathcal{A}_{x_k}a^{jk}l_{x_j}=&\,\,2\lambda^3\mu^4\varphi^3\Bigg(\sum_{j,k=1}^Na^{jk}\psi_{x_j}\psi_{x_k}\Bigg)^2+\lambda^2\varphi^2O(\mu^4)\\
\label{5.14015}&\,+\lambda^3\varphi^3O(\mu^3)+\lambda^2T\varphi^3O(\mu^2e^{2\mu\vert\psi\vert_\infty}),
\end{aligned}
\end{align}
and 
\begin{align}
\begin{aligned}
\sum_{j,k=1}^N \mathcal{A}\big(a^{jk}_{x_k}l_{x_j}+a^{jk}l_{x_jx_k}\big)=&\,\lambda^3\mu^4\varphi^3\Bigg(\sum_{j,k=1}^Na^{jk}\psi_{x_j}\psi_{x_k}\Bigg)^2+\lambda^3\varphi^3O(\mu^3)\\
\label{5.15012}&+\lambda^2\varphi^2O(\mu^4)+\lambda^2T\varphi^3O(\mu^2e^{2\mu\vert\psi\vert_\infty}).
\end{aligned}
\end{align}
From \eqref{5.16170123}, \eqref{5.14015} and \eqref{5.15012}, we deduce 
\begin{align} 
\begin{aligned}
\sum_{j,k=1}^N (\mathcal{A}a^{jk}l_{x_j})_{x_k}=&\,\,3\lambda^3\mu^4\varphi^3\Bigg(\sum_{j,k=1}^Na^{jk}\psi_{x_j}\psi_{x_k}\Bigg)^2+\lambda^2\varphi^2O(\mu^4)\\\label{5.181801}&
+\lambda^3\varphi^3O(\mu^3)+\lambda^2T\varphi^3O(\mu^2e^{2\mu\vert\psi\vert_\infty}).
\end{aligned}
\end{align}
Similarly, one can prove that 
\begin{align}
\begin{aligned}
\mathcal{A}_t&\,=\lambda^2\varphi^2O(\mu^2)+2\lambda^2\mu^2\varphi\varphi_t\sum_{j,k=1}^Na^{jk}\psi_{x_j}\psi_{x_k}+\lambda T\varphi^2 O(\mu^2)\\
\label{5.191901}&\quad\;+\lambda\varphi O(\mu^2)+\lambda T^2\varphi^3O(e^{2\mu\vert\psi\vert_\infty}).
\end{aligned}
\end{align}
By taking a large $\lambda\geq CT^2$ in \eqref{5.191901}, one absorbs the third term by using the second term, hence we get
\begin{align}
\mathcal{A}_t\label{5.191902}&\,=\lambda^2\varphi^2O(\mu^2)+\lambda^2T\varphi^3O(\mu^2)+\lambda\varphi O(\mu^2)+\lambda T^2\varphi^3O(e^{2\mu\vert\psi\vert_\infty}).
\end{align}
Combining \eqref{5.141401}, \eqref{5.151501}, \eqref{5.181801} and \eqref{5.191902}, we obtain
\begin{align}
\begin{aligned}
\mathcal{B}=&\,2\lambda^3\mu^4\varphi^3\Bigg(\sum_{j,k=1}^Na^{jk}\psi_{x_j}\psi_{x_k}\Bigg)^2+\lambda^3\varphi^3O(\mu^3)+\lambda^2\varphi^2O(\mu^4)\\
\label{5.1918012}&+\lambda^2T\varphi^3O(\mu^2e^{2\mu\vert\psi\vert_\infty})+\lambda T^2\varphi^3O(e^{2\mu\vert\psi\vert_\infty})+\lambda\varphi O(\mu^4).
\end{aligned}
\end{align}
Then \eqref{5.1918012} implies the estimate of $\mathcal{B}$.

For the estimate of $c^{jk}$: See that
$$c^{jk}=\sum_{j',k'=1}^N \Bigg[2a^{jk'}(a^{j'k}l_{x_{j'}})_{x_{k'}}-(a^{jk}a^{j'k'})_{x_{k'}}l_{x_{j'}}+a^{jk}a^{j'k'}l_{x_{j'}x_{k'}}\Bigg]+\frac{a^{jk}_t}{2}.$$
It follows
$$c^{jk}=\sum_{j',k'=1}^N \Bigg[2a^{jk'}a^{j'k}l_{x_{j'}x_{k'}}+ a^{jk} a^{j'k'} l_{x_{j'}x_{k'}}+2a^{jk'}a^{j'k}_{x_{k'}} l_{x_{j'}} -(a^{jk} a^{j'k'})_{x_{k'}} l_{x_{j'}}\Bigg]+\frac{a^{jk}_t}{2}.$$
Then for all $\xi=(\xi_1,...,\xi_N)\in\mathbb{R}^N$, we obtain
\begin{align*}
\sum_{j,k=1}^N c^{jk}\xi_j\xi_k&\,=2\lambda\mu^2\varphi\Bigg(\sum_{j,k=1}^N a^{jk}\psi_{x_j}\xi_{k}\Bigg)^2+\lambda\mu^2\varphi\Bigg(\sum_{j,k=1}^N a^{jk}\psi_{x_j}\psi_{x_k}\Bigg)\Bigg(\sum_{j,k=1}^N a^{jk}\xi_{j}\xi_{k}\Bigg)\\
&\,\,\;\;\;\;+\lambda\varphi\vert\xi\vert^2O(\mu)\\
&\,\,\geq [\beta^2\lambda\mu^2\varphi\vert\nabla\psi\vert^2+\lambda\varphi O(\mu)]\vert\xi\vert^2.
\end{align*}
This concludes the proof of Theorem \ref{thm5.2}.
\end{proof} 

We are now in a position to prove Lemma \ref{lm2.2}.
\begin{proof}[Proof of Lemma \ref{lm2.2}]
We borrow some ideas from the proof of \cite[Theorem 9.39]{lu2021mathematical}. Let $\psi$ be the function given by Lemma \ref{lm2.1} with $G_1$ being any fixed nonempty open subset of $G$ such that $G_1\Subset G_0$. Integrating the identity \eqref{5.121301} with $z$ being the solution of equation \eqref{202.3} (with $F\equiv0$), taking the expectation on both sides, we arrive at
\begin{align}\label{5.212103}
\begin{aligned}
&2\mathbb{E}\int_Q \sum_{j,k=1}^N c^{jk}h_{x_j}h_{x_k}dxdt+\mathbb{E}\int_Q\mathcal{B}h^2 dxdt\\
&\leq \mathbb{E}\int_Q \theta^2F_0^2 dxdt+\mathbb{E}\int_Q \theta^2\mathcal{A}Z^2 dxdt.
\end{aligned}
\end{align}
Now, using Theorem \ref{thm5.2}, we get for a large $\mu\geq C$ and $\lambda\geq C(T+T^2)$ 
\begin{align}
\begin{aligned}
&\,2\mathbb{E}\int_0^T\int_{G\setminus G_1} \sum_{j,k=1}^N c^{jk}h_{x_j}h_{x_k}dxdt+\mathbb{E}\int_0^T\int_{G\setminus G_1} \mathcal{B}h^2 dxdt\\
&\geq 2\beta^2\mathbb{E}\int_0^T\int_{G\setminus G_1} \Big[\big(\lambda\mu^2\varphi \min_{x\in G\setminus G_1}\vert\nabla\psi\vert^2+\lambda\varphi O(\mu)\big)\vert\nabla h\vert^2\\
&\quad+\big(\lambda^3\mu^4\varphi^3 \min_{x\in G\setminus G_1}\vert\nabla\psi\vert^4+\lambda^3\varphi^3O(\mu^3)+\lambda^2\varphi^2O(\mu^4)\\
\label{5.212201}&\quad+\lambda\varphi O(\mu^4)+\lambda^2T\varphi^3O(\mu^2e^{2\mu\vert\psi\vert_\infty})\big)h^2\Big] dxdt.
\end{aligned}
\end{align}
Then from \eqref{5.212201}, we have for a large $\mu\geq C$ and $\lambda\geq C(e^{2\mu\vert\psi\vert_\infty}T+T^2)$ that
\begin{align}
\begin{aligned}
&\,2\mathbb{E}\int_0^T\int_{G\setminus G_1} \sum_{j,k=1}^N c^{jk}h_{x_j}h_{x_k}dxdt+\mathbb{E}\int_0^T\int_{G\setminus G_1} \mathcal{B}h^2 dxdt\\
\label{5.232404}&\geq C\lambda\mu^2\mathbb{E}\int_0^T\int_{G\setminus G_1} \varphi(\vert\nabla h\vert^2+\lambda^2\mu^2\varphi^2h^2)dxdt.
\end{aligned}
\end{align}
It is easy to see that
\begin{equation}\label{5.242505}
\frac{1}{C}\theta^2(\vert\nabla z\vert^2+\lambda^2\mu^2\varphi^2z^2)\leq \vert\nabla h\vert^2+\lambda^2\mu^2\varphi^2h^2\leq C\theta^2(\vert\nabla z\vert^2+\lambda^2\mu^2\varphi^2z^2).
\end{equation}
From \eqref{5.212103}, \eqref{5.232404} and \eqref{5.242505}, we conclude that
\begin{align}
\begin{aligned}
&\,\lambda\mu^2\mathbb{E}\int_Q \theta^2\varphi(\vert\nabla z\vert^2+\lambda^2\mu^2\varphi^2z^2) dxdt\\
&=\lambda\mu^2\mathbb{E}\int_0^T\Bigg(\int_{G\setminus G_1} + \int_{G_1}\Bigg) \theta^2\varphi(\vert\nabla z\vert^2+\lambda^2\mu^2\varphi^2z^2) dxdt\\
&\,\leq C\Bigg[\mathbb{E}\int_Q \theta^2F_0^2 dxdt+\lambda^2\mu^2\mathbb{E}\int_Q \theta^2\varphi^2Z^2 dxdt\Bigg]\\
\label{5.212101}&\quad\;+ \lambda\mu^2\mathbb{E}\int_0^T\int_{G_1}\theta^2\varphi(\vert\nabla z\vert^2+\lambda^2\mu^2\varphi^2z^2) dxdt,
\end{aligned}
\end{align}
for $\mu\geq C$ and $\lambda\geq C(e^{2\mu\vert\psi\vert_\infty}T+T^2)$.

Now, to absorb the gradient term on the right-hand side of \eqref{5.212101}, we proceed exactly as in the proof of \cite[Theorem 9.39]{lu2021mathematical} by choosing a cut-off function $\zeta\in C^\infty_0(G_0;[0,1])$ with $\zeta\equiv1$ in $G_1$, then computing $d(\theta^2\varphi z^2)$ and integrating the equality on $Q$, we get
\begin{equation}\label{5.222201}
\mathbb{E}\int_0^T\int_{G_1}\theta^2\varphi\vert\nabla z\vert^2dxdt\leq C\mathbb{E}\int_{Q_0}\theta^2\Bigg(\frac{1}{\lambda^2\mu^2}F_0^2+\lambda^2\mu^2\varphi^3z^2\Bigg) dxdt.
\end{equation}
Finally, combining \eqref{5.212101} and \eqref{5.222201}, we deduce the desired Carleman inequality \eqref{2.061}. This concludes the proof of Lemma \ref{lm2.2}.
\end{proof}

\end{document}